\documentclass[10pt, twoside, leqno]{article}

\usepackage{amssymb}
\usepackage{amsmath}
\usepackage{amsthm}
\usepackage{color}
\usepackage{mathrsfs}

\allowdisplaybreaks

\def\rr{{\mathbb R}}
\def\rn{{\mathbb{R}^n}}
\def\urn{\mathbb{R}_+^{n+1}}
\def\zz{{\mathbb Z}}
\def\cc{{\mathbb C}}

\def\nn{{\mathbb N}}

\def\ca{{\mathcal A}}

\def\ccc{{\mathcal C}}

\def\cf{{\mathcal F}}

\def\cl{{\mathcal L}}
\def\cm{{\mathcal M}}

\def\cp{{\mathcal P}}
\def\cq{{\mathcal Q}}
\def\mr{{\mathcal R}}
\def\cs{{\mathcal S}}

\def\fz{\infty }
\def\az{\alpha}

\def\ez{\epsilon}
\def\gz{{\gamma}}

\def\lz{\lambda}

\def\oz{{\omega}}

\def\vz{\varphi}
\def\vez{\varepsilon}

\def\lf{\left}
\def\r{\right}

\def\hs{\hspace{0.25cm}}
\def\ls{\lesssim}

\def\noz{\nonumber}
\def\wz{\widetilde}
\def\wh{\widehat}

\def\gfz{\genfrac{}{}{0pt}{}}

\def\vlp{{L^{p(\cdot)}(\rn)}}
\def\vhs{H^{p(\cdot)}(\rn)}
\def\cps{\cl_{q,p(\cdot),s}(\rn)}
\def\cpss{\cl_{1,p(\cdot),s}(\rn)}

\newtheorem{theorem}{Theorem}[section]
\newtheorem{lemma}[theorem]{Lemma}
\newtheorem{corollary}[theorem]{Corollary}
\newtheorem{proposition}[theorem]{Proposition}

\theoremstyle{definition}
\newtheorem{remark}[theorem]{Remark}
\newtheorem{definition}[theorem]{Definition}
\renewcommand{\appendix}{\par
   \setcounter{section}{0}%
   \setcounter{subsection}{0}%
   \setcounter{subsubsection}{0}%
   \gdef\thesection{\@Alph\c@section}%
   \gdef\thesubsection{\@Alph\c@section.\@arabic\c@subsection}%
   \gdef\theHsection{\@Alph\c@section.}%
   \gdef\theHsubsection{\@Alph\c@section.\@arabic\c@subsection}%
   \csname appendixmore\endcsname
 }

\numberwithin{equation}{section}

\begin{document}

\arraycolsep=1pt

\title{\bf\Large Intrinsic Square Function Characterizations of Hardy Spaces
with Variable Exponents
\footnotetext{\hspace{-0.35cm} \emph{Communicated by} Rosihan M. Ali, Dato'.
\endgraf \emph{Received:} January 28, 2014; \emph{Revised:} March 26, 2014. }}
\author{\sc Ciqiang Zhuo, Dachun Yang\footnote{Corresponding author}\ \ and Yiyu Liang\\
{\small School of Mathematical Sciences, Beijing Normal University,}\\
{\small Laboratory of Mathematics and Complex Systems, Ministry of Education,}\\
{\small Beijing 100875, People's Republic of China}\\
{\small\noindent cqzhuo@mail.bnu.edu.cn, dcyang@bnu.edu.cn, yyliang@mail.bnu.edu.cn}}
\date{}
\maketitle

\vspace{-0.5cm}

\begin{center}
\begin{minipage}{13cm}
{\small {\bf Abstract}\quad
Let $p(\cdot):\ \mathbb R^n\to(0,\infty)$ be a measurable function satisfying some
decay condition and some locally log-H\"older continuity.
In this article, via first establishing characterizations
of the variable exponent Hardy space $H^{p(\cdot)}(\mathbb R^n)$ in terms of
the Littlewood-Paley $g$-function, the Lusin area function
and the $g_\lambda^\ast$-function, the authors then obtain
its intrinsic square function
characterizations including the intrinsic Littlewood-Paley $g$-function,
the intrinsic Lusin area function and the intrinsic $g_\lambda^\ast$-function.
The $p(\cdot)$-Carleson measure
characterization for the dual space of $H^{p(\cdot)}(\mathbb R^n)$,
the variable exponent Campanato space $\mathcal{L}_{1,p(\cdot),s}(\mathbb R^n)$,
in terms of the intrinsic function is also presented.

\vspace{0.5cm}

2010 Mathematics Subject Classification: Primary: 42B25,
Secondary: 42B30, 42B35, 46E30

\vspace{0.5cm}

Keywords and phrases: Hardy space, variable exponent,
intrinsic square function, Carlson measure, atom.}
\end{minipage}
\end{center}

\vspace{0.1cm}

\section{Introduction\label{s1}}

Variable exponent Lebesgue spaces are a generalization of the
classical $L^p(\rn)$ spaces,
in which the constant exponent $p$ is replaced by an exponent function
$p(\cdot): \rn\to(0,\fz)$, namely, they consist of all functions $f$ such that
$\int_\rn|f(x)|^{p(x)}\,dx<\fz$.
These spaces were introduced by Birnbaum-Orlicz \cite{bo31} and Orlicz
\cite{ol32},
and widely used in the study of harmonic analysis as well as partial
differential equations;
see, for example,
\cite{aa13,ai00,cfmp06,cdh11,cfbook,dhhm03,din04,dhr11,sd13,io02,mw08,wt13,wl13,y14}.
For a systematic research about the variable exponent Lebesgue space, we refer
the reader to \cite{cfbook,dhr11}

Recently, Nakai and Sawano \cite{ns12} extended the theory of variable
Lebesgue spaces via studying the Hardy spaces with variable exponents on $\rn$,
and Sawano in \cite{sawa13} further gave more applications of these
variable exponent Hardy spaces. Independently, Cruz-Uribe and Wang in
\cite{cw13} also investigated
the variable exponent Hardy space with some weaker conditions than those used in
\cite{ns12}, which also extends the theory of variable exponent Lebesgue spaces.
Recall that the classical Hardy spaces $H^p(\rn)$ with $p\in(0,1]$
on the Euclidean space $\rn$
and their duals are well studied (see, for example, \cite{fs72,{stein93}})
and have been playing
an important and fundamental role in various fields of analysis such as
harmonic analysis
and partial differential equations; see, for example, \cite{cg77,{muller94}}.

On the other hand, the study of the intrinsic square function
on function spaces, including Hardy spaces, has recently
attracted many attentions.
To be precise, Wilson \cite{wilson07} originally introduced
intrinsic square functions,
which can be thought of as ``grand maximal" square functions
of C. Fefferman and E. M. Stein from \cite{fs72},
to settle a conjecture proposed by R. Fefferman and
E. M. Stein on the boundedness of the Lusin area function $S(f)$ from
the weighted Lebesgue space $L_{\cm(v)}^2(\rn)$
to the weighted Lebesgue space $L^2_v(\rn)$, where $0\le v\in L_{\rm loc}^1(\rn)$
and $\cm$ denotes the Hardy-Littlewood maximal function.
The boundedness of these intrinsic square functions on the weighted
Lebesgue spaces $L^p_{\omega}(\rn)$, when $p\in(1,\fz)$
and $\oz$ belongs to Muckenhoupt weights $A_p(\rn)$,
was proved by Wilson \cite{w08}.
The intrinsic square functions dominate all square functions of the form
$S(f)$ (and the classical ones as well), but are not essentially bigger than any
one of them. Similar to the Fefferman-Stein and the Hardy-Littlewood
 maximal functions, their generic natures make them pointwise
equivalent to each other and extremely
easy to work with. Moreover, the intrinsic Lusin area function has the distinct
advantage of being pointwise comparable at different cone openings, which is a
property long known not to hold true for the classical Lusin area function;
see Wilson \cite{wilson07,w08,wilson10,wilson11} and also Lerner
\cite{lerner11,lerner13}.

Later, Huang and Liu in \cite{hl10} obtain the intrinsic square function
characterizations of the weighted Hardy space $H^1_\omega(\rn)$ under the
additional assumption that $f\in L_\omega^1(\rn)$,
which was further generalized to the weighted Hardy space $H^p_\omega(\rn)$
with $p\in( n/{(n+\alpha)},1)$ and $\alpha\in(0,1)$ by Wang and Liu in
\cite{wangliu}, under another additional assumption.
Very recently, Liang and Yang in \cite{lyts13} established the
$s$-order intrinsic
square function characterizations of the
Musielak-Orlicz Hardy space $H^\vz(\rn)$, which was introduced by Ky \cite{ky11}
and generalized both the Orlicz-Hardy space (see, for example,
\cite{janson80,{viviani87}}) and
the weighted Hardy space (see, for example, \cite{gc79,{st89}}), in terms of the
intrinsic Lusin area function, the intrinsic $g$-function and the intrinsic
$g_\lz^\ast$-function with the best known range
$\lz\in(2+2(\az+s)/n,\fz)$. More applications of such intrinsic square functions
were also given by Wilson \cite{wilson10,wilson11} and
Lerner \cite{lerner11,{lerner13}}.

Motivated by \cite{lyts13}, in this article,
we establish intrinsic square function
characterizations of the variable exponent Hardy space $\vhs$ introduced by Nakai
and Sawano in \cite{ns12},
including the intrinsic Littlewood-Paley $g$-function,
the intrinsic Lusin area function and the
intrinsic $g_\lz^\ast$-function by first
obtaining characterizations of $\vhs$ via the Littlewood-Paley $g$-function,
the Lusin area function and the $g_\lz^\ast$-function.
We also establish
the $p(\cdot)$-Carleson measure characterization for the dual space of
$H^{p(\cdot)}(\rn)$, the variable exponent Campanato space
$\mathcal{L}_{1,p(\cdot),s}(\mathbb R^n)$ in \cite{ns12},
in terms of the intrinsic square function.

To state the results, we begin with some notation.
In what follows, for a measurable function
$p(\cdot):\ \rn\to(0,\fz)$ and a measurable set $E$ of $\rn$, let
$$p_-(E):=\mathop{\rm ess\,inf}\limits_{x\in E}p(x)\quad {\rm and}\quad
p_+(E):=\mathop{\rm ess\,sup}\limits_{x\in E}p(x).$$
For simplicity, we let $p_-:=p_-(\rn)$, $p_+:=p_+(\rn)$ and
$p^\ast:=\min\{p_-,1\}$.
Denote by $\cp(\rn)$ the \emph{collection of all measurable functions
$p(\cdot):\ \rn\to(0,\fz)$ satisfying $0<p_-\le p_+<\fz$}.

For $p(\cdot)\in\cp(\rn)$, the \emph{space} $L^{p(\cdot)}(\rn)$
is defined to be the set of all measurable functions such that
$$\|f\|_{\vlp}:=\inf\lf\{\lz\in(0,\fz):\ \int_\rn\lf[\frac{|f(x)|}
{\lz}\r]^{p(x)}\,dx\le1\r\}<\fz.$$

\begin{remark}\label{r-vlp}
It was pointed out in \cite[p.\,3671]{ns12}
(see also \cite[Theorem 2.17]{cfbook})
that the follows hold true:
\begin{itemize}

\item[(i)] $\|f\|_{\vlp}\ge0$, and $\|f\|_{\vlp}=0$ if and only if $f(x)=0$
for almost every $x\in\rn$;

\item[(ii)] $\|\lz f\|_{\vlp}=|\lz|\|f\|_{\vlp}$ for any $\lz\in\cc$;

\item[(iii)] $\|f+g\|_{\vlp}^\ell\le \|f\|_{\vlp}^\ell+\|g\|_{\vlp}^\ell$ for all
$\ell\in(0,p^\ast]$;

\item[(iv)] for all measurable functions $f$ with $\|f\|_{\vlp}\neq0$,
$\int_\rn[{|f(x)|}/{\|f\|_{\vlp}}]^{p(x)}\,dx=1$.
\end{itemize}
\end{remark}
A function $p(\cdot)\in\cp(\rn)$ is said to satisfy the
\emph{locally log-H\"older continuous condition}
if there exists a positive constant $C$ such that, for all $x,\ y\in\rn$
and $|x-y|\le1/2$,
\begin{equation}\label{ve1}
|p(x)-p(y)|\le \frac{C}{\log(1/|x-y|)},
\end{equation}
and $p(\cdot)$ is said to satisfy the \emph{decay condition}
if there exist positive constants $C_\fz$ and $p_\fz$
such that, for all $x\in\rn$,
\begin{equation}\label{ve2}
|p(x)-p_\fz|\le \frac{C_\fz}{\log(e+|x|)}.
\end{equation}

In the whole article, we denote by $\cs(\rn)$ the \emph{space of all
Schwartz functions} and by $\cs'(\rn)$ its \emph{topological dual space}.
Let $\cs_\fz(\rn)$ denote the \emph{space of all Schwartz functions $\vz$}
satisfying $\int_\rn\vz(x)x^\beta\,dx=0$ for all multi-indices
$\beta\in\zz_+^n:=(\{0,1,\ldots\})^n$
and $\cs_\fz'(\rn)$ its \emph{topological dual space}.
For $N\in\nn:=\{1,2,\dots\}$, let
\begin{equation}\label{fnd}
\cf_N(\rn):=\lf\{\psi\in\cs(\rn):\ \sum_{\beta\in\zz_+^n,\,|\beta|\le N}
\sup_{x\in\rn}(1+|x|)^N|D^\beta\psi(x)|\le1\r\},
\end{equation}
where, for $\beta:=(\beta_1,\dots,\beta_n)\in\zz_+^n$,
$|\beta|:=\beta_1+\cdots+\beta_n$ and
$D^\beta:=(\frac\partial{\partial x_1})^{\beta_1}
\cdots(\frac\partial{\partial x_n})^{\beta_n}$.
Then, for all $f\in\cs'(\rn)$, the \emph{grand maximal function} $f_{N,+}^\ast$
of $f$ is defined by setting, for all $x\in\rn$,
$$f_{N,+}^\ast(x):=\sup\lf\{|f\ast\psi_t(x)|:\
t\in(0,\fz)\ {\rm and}\ \psi\in\cf_N(\rn)\r\},$$
where, for all $t\in(0,\fz)$ and $\xi\in\rn$, $\psi_t(\xi):=t^{-n}\psi(\xi/t)$.

For any measurable set $E\subset\rn$ and $r\in(0,\fz)$, let $L^r(E)$
be the \emph{set of all measurable functions} $f$ such that
$\|f\|_{L^r(E)}:=\lf\{\int_E|f(x)|^r\,dx\r\}^{1/r}<\fz$.
For $r\in(0,\fz)$, denote by $L_{\rm loc}^r(\rn)$ the \emph{set of all
$r$-locally integrable functions}
on $\rn$.
Recall that the \emph{Hardy-Littlewood maximal operator}
$\cm$ is defined by setting,
for all $f\in L_{\rm loc}^1(\rn)$ and $x\in\rn$,
$$\cm(f)(x):=\sup_{B\ni x}\frac1{|B|}\int_B |f(y)|\,dy,$$
where the supremum is taken over all balls $B$ of $\rn$ containing $x$.

Now we recall the notion of the Hardy space with variable exponent, $\vhs$,
introduced by Nakai and Sawano in \cite{ns12}. For simplicity, we also call
$\vhs$ the variable exponent Hardy space.
\begin{definition}\label{d-vhs}
Let $p(\cdot)\in\cp(\rn)$ satisfy \eqref{ve1} and \eqref{ve2},
and
\begin{equation}\label{1nx}
N\in\lf(\frac n{p_-}+n+1,\fz\r)\cap\nn.
\end{equation}
The \emph{Hardy space with variable exponent}
$p(\cdot)$, denoted by $H^{p(\cdot)}(\rn)$, is defined to be the set of all
$f\in\cs'(\rn)$ such that
$f_{N,+}^\ast\in \vlp$ with the \emph{quasi-norm} $\|f\|_{H^{p(\cdot)}(\rn)}:=
\|f_{N,+}^\ast\|_{\vlp}$.
\end{definition}
\begin{remark}
\begin{itemize}
\item[(i)] Independently, Cruz-Uribe and Wang in \cite{cw13}
introduced the variable exponent Hardy space, denoted by $\wz H^{p(\cdot)}$,
in the following way:
Let $p(\cdot)\in\cp(\rn)$ satisfy that there
exist $p_0\in(0,p_-)$ and a positive constant $C$,
only depending
on $n,\ p(\cdot)$ and $p_0$, such that
\begin{equation}\label{max}
\|\cm(f)\|_{L^{p(\cdot)/p_0}(\rn)}
\le C\|f\|_{L^{p(\cdot)/p_0}(\rn)}.
\end{equation}
If $N\in( n/{p_0}+n+1,\fz)$,
then the \emph{variable exponent Hardy space $\wz H^{p(\cdot)}$}
is defined to be the set of all $f\in\cs'(\rn)$ such that
$f_{N,+}^\ast\in \vlp$. In \cite[Theorem 3.1]{cw13}, it was shown that
the space $\wz H^{p(\cdot)}$ is independent of the choice of
$N\in( n/{p_0}+n+1,\fz)$.

\item[(ii)] We point out that, in \cite[Theorem 3.3]{ns12}, it was proved that
the space $\vhs$ is independent of $N$ as long as $N$ is sufficiently large.
Although the range of $N$ is not presented explicitly in \cite[Theorem 3.3]{ns12},
by the proof of \cite[Theorem 3.3]{ns12}, we see that
$N$ as in \eqref{1nx} does the work.
\end{itemize}
\end{remark}

Let $\phi\in\cs(\rn)$ be a radial real-valued function satisfying
\begin{equation}\label{function1}
{\rm supp}\,\wh\phi\subset\{\xi\in\rn:\ 1/2\le|\xi|\le2\}
\end{equation}
and
\begin{equation}\label{function2}
|\wh\phi(\xi)|\ge C\ {\rm if}\ 3/5\le|\xi|\le5/3,
\end{equation}
where $C$ denotes a positive constant independent of $\xi$ and,
for all $\phi\in\cs(\rn)$,
$\wh\phi$ denotes its \emph{Fourier transform}.
Obviously, $\phi\in\cs_\fz(\rn)$.
Then, for all $f\in\cs_\fz'(\rn)$,
the \emph{Littlewood-Paley $g$-function}, the \emph{Lusin area function}
and the \emph{$g_\lz^\ast$-function} with $\lz\in(0,\fz)$ of $f$ are,
respectively,
defined by setting, for all
$x\in\rn$,
$$g(f)(x):=\lf\{\int_0^\fz|f\ast \phi_t(x)|^2\frac{dt}t\r\}^{1/2},$$
\begin{equation*}
S(f)(x)
:=\lf\{\int_0^\fz\int_{\{y\in\rn:\ |y-x|<t\}}|\phi_t\ast f(y)|^2\frac{dy\,dt}
{t^{n+1}}\r\}^{1/2}
\end{equation*}
and
\begin{equation}\label{1gx}
g_\lz^\ast(f)(x):=\lf\{\int_0^\fz\int_\rn\lf(\frac t{t+|x-y|}\r)^{\lz n}
|\phi_t\ast f(y)|^2\,\frac{dy\,dt}{t^{n+1}}\r\}^{\frac12}.
\end{equation}
For all $f\in\cs_\fz'(\rn)$ and $\phi\in\cs(\rn)$ satisfying \eqref{function1}
and
\eqref{function2}, we let,
for all $t\in(0,\fz)$, $j\in\zz$, $a\in(0,\fz)$ and $x\in\rn$,
$$(\phi_t^*f)_a(x):=\sup_{y\in\rn}\frac{|\phi_t\ast f(x+y)|}{(1+|y|/t)^a}
\quad \text{and}\quad (\phi_j^*f)_a(x):=\sup_{y\in\rn}
\frac{|\phi_j\ast f(x+y)|}{(1+2^j|y|)^a}.$$
Then, for all $f\in\cs_\fz'(\rn)$, $a\in(0,\fz)$ and $x\in\rn$,
define
$$g_{a,\ast}(f)(x):=\lf\{\int_0^\fz\lf[(\phi_t^\ast f)_a(x)\r]^2\frac{dt}t\r\}
^{1/2}\quad \mathrm{and}\quad
\sigma_{a,\ast}(f)(x):=\lf\{\sum_{j\in\zz}\lf[(\phi_j^\ast f)_a(x)\r]
^2\r\}^{1/2}.$$

The following conclusion is the first main result of this article.

\begin{theorem}\label{t-equivalen}
Let $p(\cdot)\in\cp(\rn)$ satisfy \eqref{ve1} and \eqref{ve2}.
Then $f\in\vhs$ if and only if $f\in\cs_\fz'(\rn)$ and $S(f)\in\vlp$;
moreover, there exists a positive constant $C$, independent of $f$, such that
$C^{-1}\|S(f)\|_{\vlp}\le\|f\|_{\vhs}\le C\|S(f)\|_{\vlp}.$

The same is true if $S(f)$ is replaced, respectively, by $g(f)$, $g_{a,\ast}(f)$
and $\sigma_{a,\ast}(f)$ with $a\in(n/\min\{p_-,2\},\fz)$.
\end{theorem}

\begin{corollary}\label{c-equivalent}
Let $p(\cdot)\in\cp(\rn)$ and $\lz\in(1+2/{\min\{2,p_-\}},\fz)$. Then
$f\in\vhs$ if and only if $f\in\cs_\fz'(\rn)$ and $g_\lz^\ast(f)\in\vlp$;
moreover, there exists a positive constant $C$, independent of $f$, such that
$C^{-1}\|g_\lz^\ast(f)\|_{\vlp}\le\|f\|_{\vhs}\le C\|g_\lz^\ast(f)\|_{\vlp}.$
\end{corollary}

\begin{remark}\label{r-equivalen}
\begin{itemize}
\item[(i)] We point out that the conclusion of Theorem \ref{t-equivalen} is understood in
the following sense:
if $f\in\vhs$, then $f\in \cs_\fz'(\rn)$ and there exists a positive
constant $C$ such that, for all $f\in\vhs$,
$\|S(f)\|_{\vlp}\le C\|f\|_{\vhs}$;
conversely, if $f\in\cs_\fz'(\rn)$ and $S(f)\in\vlp$, then there exists a unique
extension $\wz f \in\cs'(\rn)$ such that, for all $h\in\cs_\fz(\rn)$,
$\langle \wz f,h\rangle=\langle f,h\rangle$ and
$\|\wz f\|_{\vhs}\le C\|S(f)\|_{\vlp}$
with $C$ being a positive constant independent of $f$.
In this sense, we identify $f$ with $\wz f$.

\item[(ii)] Recall that, Hou et al. \cite{hyy} characterized the
Musielak-Orlicz Hardy space $H^{\vz}(\rn)$, which was introduced by Ky \cite{ky11},
via the Lusin area function, and Liang et al. \cite{lhy} established
the Littlewood-Paley $g$-function and the $g_\lz^\ast$-function characterizations
of $H^{\vz}(\rn)$. Observe that, when
\begin{equation}\label{vz}
\vz(x,t):=t^{p(x)}\quad\mathrm{for\ all}\quad x\in\rn\quad
\mathrm{and}\quad t\in[0,\fz),
\end{equation}
then $H^\vz(\rn)=H^{p(\cdot)}(\rn)$. However, a general
Musielak-Orlicz function $\vz$ satisfying all the assumptions
in \cite{ky11} (and hence \cite{hyy,lhy}) may not have the form
as in \eqref{vz}. On the other hand, it was proved in
\cite[Remark 2.23(iii)]{yyz13} that there exists an exponent function
$p(\cdot)$ satisfying \eqref{ve1} and \eqref{ve2},
but $t^{p(\cdot)}$ is not a uniformly Muckenhoupt weight,
which was required in \cite{ky11} (and hence \cite{hyy,lhy}).
Thus, the Musielak-Orlicz Hardy space $H^\vz(\rn)$
in \cite{ky11} (and hence in \cite{hyy,lhy}) and the variable
exponent Hardy space $H^{p(\cdot)}(\rn)$ in \cite{ns12}
(and hence in the present article) can not cover each other.

Moreover, Liang et al. \cite[Theorem 4.8]{lhy} established the $g_\lz^\ast$-function
characterization of the Musielak-Orlicz Hardy space $H^\vz(\rn)$
with the best known range for $\lz$.
In particular, in the case of the classical Hardy space $H^p(\rn)$,
$\lz\in(2/{\min\{p,2\}},\fz)$; see, for example, \cite[p.\,221, Corollary 7.4]{fs82}
and \cite[p.\,91, Theorem 2]{st70}.
However, it is still unclear whether the $g_\lz^\ast$-function,
when $\lz\in(2/{\min\{p_-,2\}},1+2/{\min\{p_-,2\}}]$,
can characterize $\vhs$ or not,
since the method used in \cite[Theorem 4.8]{lhy} strongly depends on the properties of
uniformly Muckenhoupt weights, which are not satisfied by $t^{p(\cdot)}$.

Indeed, a key fact that used in the proof of \cite[Theorem 4.8]{lhy},
which may not hold in the present setting, is that,
if $\vz$ is a Musielak-Orlicz function as in \cite{lhy}, then
there exists a positive constant $C$ such that,
for all $\lz\in(0,\fz)$, $\alpha\in(0,1)$ and measurable set $E\subset\rn$,
\begin{equation*}
\int_{U(E;\alpha)}\vz(x,\lz)\,dx\le C\alpha^{nq}\int_E\vz(x,\lz)\,dx,
\end{equation*}
where $U(E;\alpha):=\{x\in\rn:\ \cm(\chi_E)(x)>\alpha\}$ and $q\in[1,\fz)$
is the uniformly Muckenhoupt weight index of $\vz$.
To see this, following \cite[Example 1.3]{ns12}, for all $x\in\rr$, let
$$p(x):=\max\lf\{1-e^{3-|x|},\min\lf(6/5,\max\{1/2,3/2-x^2\}\r)\r\}.$$
Then $p(\cdot)$ satisfies \eqref{ve1} and \eqref{ve2}.
Now, let $E:=(1,2)$, then, for all $x\in\rr$,
$$\cm(\chi_E)(x)=\chi_E(x)+\frac1{1+2|x-3/2|}\chi_{\rr\setminus E}(x).$$
It is easy to see that, for all $\lz\in(0,\fz)$,
$\int_E\lz^{p(x)}\,dx=\lz^{1/2}$
and
$$\int_{U(E;1/11)}\lz^{p(x)}\,dx=\int_{-\frac72}^{-\frac{13}2}\lz^{p(x)}\,dx
>\int_E\lz^{p(x)}\,dx+\int_{-\frac12}^{\frac12}\lz^{p(x)}\,dx=\lz^{1/2}+\lz^{6/5}.$$
Thus, we find that
$$\lim_{\lz\to\fz}\frac{\int_{U(E;1/11)}\lz^{p(x)}\,dx}{\int_E\lz^{p(x)}\,dx}=\fz,$$
which implies that
there does not exist a positive constant $C$, independent of $\lz$,
such that,
\begin{equation*}
\int_{U(E;1/11)}\lz^{p(x)}\,dx\le C\int_E\lz^{p(x)}\,dx.
\end{equation*}
Thus, the method used in the proof of \cite[Theorem 4.8]{lhy} is not suitable for the present setting.
\end{itemize}
\end{remark}

For any $s\in\zz_+$, $C^s(\rn)$ denotes the \emph{set of all functions having
continuous classical derivatives up to order not more than $s$}.
For $\alpha\in(0,1]$ and $s\in\zz_+$, let $\ccc_{\alpha,s}(\rn)$ be the
\emph{family of functions $\phi\in C^s(\rn)$} such that
supp\,$\phi\subset\{x\in\rn:\ |x|\le1\}$, $\int_\rn\phi(x)x^\gamma\,dx=0$
for all $\gamma\in\zz_+^n$ and $|\gamma|\le s$, and,
for all $x_1,\ x_2\in\rn$ and $\nu\in\zz_+^n$ with $|\nu|=s$,
\begin{equation}\label{condition1}
|D^\nu\phi(x_1)-D^\nu\phi(x_2)|\le|x_1-x_2|^\alpha.
\end{equation}

For all $f\in L_{\rm loc}^1(\rn)$ and $(y,t)\in\rr_+^{n+1}:=\rn\times(0,\fz)$,
let
$$A_{\alpha,s}(f)(y,t):=\sup_{\phi\in\ccc_{\alpha,s}(\rn)}|f\ast \phi_t(y)|.$$
Then, the \emph{intrinsic $g$-function},
the \emph{intrinsic Lusin area integral} and
the \emph{intrinsic $g_\lz^\ast$-function} of $f$ are, respectively,
defined by setting, for all $x\in\rn$ and $\lz\in(0,\fz)$,
\begin{equation*}
g_{\alpha,s}(f)(x):=\lf\{\int_0^\fz\lf[A_{\alpha,s}(f)(x,t)\r]^2\,
\frac{dt}t\r\}^{1/2},
\end{equation*}
\begin{equation*}
S_{\alpha,s}(f)(x):=\lf\{\int_0^\fz\int_{\{y\in\rn:\ |y-x|<t\}}
\lf[A_{\alpha,s}(f)(y,t)\r]^2\,\frac{dy\,dt}{t^{n+1}}\r\}^{1/2}
\end{equation*}
and
\begin{equation*}
g_{\lz,\alpha,s}^\ast(f)(x)
:=\lf\{\int_0^\fz\int_\rn\lf(\frac t{t+|x-y|}\r)^{\lz n}
[A_{\alpha,s}(f)(y,t)]^2\,\frac{dy\,dt}{t^{n+1}}\r\}^{1/2}.
\end{equation*}

We also recall another kind of similar-looking square functions,
defined via convolutions with kernels that have unbounded supports.
For $\alpha\in(0,1]$, $s\in\zz_+$ and $\epsilon\in(0,\fz)$, let
$\ccc_{(\alpha,\epsilon),s}(\rn)$ be the \emph{family of functions}
$\phi\in C^s(\rn)$
such that, for all $x\in\rn$, $\gamma\in\zz_+^n$ and $|\gamma|\le s$,
$|D^\gamma\phi(x)|\le(1+|x|)^{-n-\epsilon}$, $\int_\rn\phi(x)x^\gamma\,dx=0$
and, for all $x_1,\ x_2\in\rn$, $\nu\in\zz_+^n$ and $|\nu|=s$,
\begin{equation}\label{intrinsic-2}
|D^\nu\phi(x_1)-D^\nu\phi(x_2)|
\le|x_1-x_2|^\alpha\lf[(1+|x_1|)^{-n-\epsilon}+(1+|x_2|)^{-n-\epsilon}\r].
\end{equation}
Remark that, in what follows, the parameter $\epsilon$ usually has to be
chosen to be large enough.
For all $f$ satisfying
\begin{equation}\label{intrinsic-1}
|f(\cdot)|(1+|\cdot|)^{-n-\epsilon}\in L^1(\rn)
\end{equation}
and $(y,t)\in\urn$, let
\begin{equation}\label{intrinsic-3}
\wz A_{(\alpha,\epsilon),s}(f)(y,t):=\sup_{\phi\in\ccc_{(\az,\ez),s}(\rn)}
|f\ast\phi_t(y)|.
\end{equation}
Then, for all $x\in\rn$ and $\lz\in(0,\fz)$, we let
\begin{equation*}
\wz g_{(\az,\ez),s}(f)(x)
:=\lf\{\int_0^\fz\lf[\wz A_{(\alpha,\epsilon),s}(f)(x,t)\r]^2\,
\frac{dt}t\r\}^{1/2},
\end{equation*}
\begin{equation*}
\wz S_{(\az,\ez),s}(f)(x):=\lf\{\int_0^\fz\int_{\{y\in\rn:\ |y-x|<t\}}
\lf[\wz A_{(\alpha,\epsilon),s}(f)(y,t)\r]^2\,\frac{dy\,dt}{t^{n+1}}\r\}^{1/2}
\end{equation*}
and
\begin{equation*}
\wz g_{\lz,(\az,\ez),s}^\ast(f)(x)
:=\lf\{\int_0^\fz\int_\rn\lf(\frac t{t+|x-y|}\r)^{\lz n}
[\wz A_{(\alpha,\epsilon),s}(f)(y,t)]^2\,\frac{dy\,dt}{t^{n+1}}\r\}^{1/2}.
\end{equation*}
These intrinsic square functions, when $s=0$, were original introduced by
Wilson \cite{wilson07}, which were further generalized to $s\in\zz_+$ by
Liang and Yang \cite{lyts13}.

In what follows, for any $r\in\zz_+$, we use $\cp_r(\rn)$ to denote
the \emph{set of all polynomials on $\rn$ with order not more than $r$}.

We now recall the notion of the Campanato space with variable exponent,
which was introduced by Nakai and Sawano in \cite{ns12}.
\begin{definition}\label{d-cps}
Let $p(\cdot)\in\cp(\rn)$, $s$ be a nonnegative integer and $q\in[1,\fz)$. Then
the \emph{Campanato space} $\cl_{q,p(\cdot),s}(\rn)$ is defined to be the set of
all $f\in L_{\rm loc}^q(\rn)$ such that
$$\|f\|_{\cps}:=\sup_{Q\subset\rn}{\frac{|Q|}{\|\chi_Q\|_{\vlp}}}
\lf[\frac1{|Q|}\int_Q|f(x)-P_Q^sf(x)|^q\,dx\r]^{\frac1q}<\fz,$$
where the supremum is taken over all cubes $Q$ of $\rn$ and $P_Q^sg$ denotes the
\emph{unique polynomial $P\in\cp_s(\rn)$} such that, for all $h\in\cp_s(\rn)$,
$\int_Q[f(x)-P(x)]h(x)\,dx=0$.
\end{definition}

Now we state the second main result of this article. Recall that $f\in\cs'(\rn)$ is said to \emph{vanish weakly at infinity},
if, for every $\phi\in\cs(\rn)$, $f\ast \phi_t\to0$ in $\cs'(\rn)$
as $t\to\fz$; see, for example, \cite[p.\,50]{fs82}.

\begin{theorem}\label{t-intrinsic}
Let $p(\cdot)\in \cp(\rn)$ satisfy \eqref{ve1}, \eqref{ve2} and $p_+\in(0,1]$.
Assume that $\alpha\in(0,1]$, $s\in\zz_+$ and
$p_-\in( n/{n+\alpha+s},1]$.
Then $f\in\vhs$ if and only if $f\in(\cl_{1,p(\cdot),s}(\rn))^\ast$,
the dual space of
$\cl_{1,p(\cdot),s}(\rn)$, $f$ vanishes weakly at infinity and
$g_{\az,s}(f)\in\vlp$;
moreover, it holds true that
$$\frac 1C\|g_{\alpha,s}(f)\|_{\vlp}\le \|f\|_{\vhs}
\le C\|g_{\alpha,s}(f)\|_{\vlp}$$
with $C$ being a positive constant independent of $f$.

The same is true if $g_{\az,s}(f)$ is replaced by $\wz g_{(\az,\ez),s}(f)$
with $\ez\in(\az+s,\fz)$.
\end{theorem}

Observe that, for all $x\in\rn$, $S_{\az,s}(f)(x)$ and
$g_{\az,s}(f)(x)$ as well as
$\wz S_{(\az,\ez),s}(f)(x)$ and $\wz g_{(\az,\ez),s}(f)(x)$
are pointwise comparable
 (see \cite[Proposition 2.4]{lyts13}),
which, together with Theorem \ref{t-intrinsic}, immediately implies
the following Corollary \ref{c-intrinsic}.

\begin{corollary}\label{c-intrinsic}
Let $p(\cdot)\in \cp(\rn)$ satisfy \eqref{ve1}, \eqref{ve2} and $p_+\in(0,1]$.
Assume that $\alpha\in(0,1]$, $s\in\zz_+$ and
$p_-\in( n/{(n+\alpha+s)},1]$.
Then $f\in\vhs$ if and only if $f\in(\cl_{1,p(\cdot),s}(\rn))^\ast$,
$f$ vanishes weakly at infinity and $S_{\az,s}(f)\in\vlp$;
moreover, it holds true that
$$\frac 1C\|S_{\alpha,s}(f)\|_{\vlp}\le \|f\|_{\vhs}
\le C\|S_{\alpha,s}(f)\|_{\vlp}$$
with $C$ being a positive constant independent of $f$.

The same is true if $S_{\az,s}(f)$ is replaced by $\wz S_{(\az,\ez),s}(f)$
with $\ez\in(\az+s,\fz)$.
\end{corollary}

\begin{theorem}\label{t-intrinsic-1}
Let $p(\cdot)\in \cp(\rn)$ satisfy \eqref{ve1}, \eqref{ve2} and $p_+\in(0,1]$.
Assume that $\alpha\in(0,1]$, $s\in\zz_+$,
$p_-\in( n/{(n+\alpha+s)},1]$ and $\lz\in(3+2(\az+s)/n,\fz)$.
Then $f\in\vhs$ if and only if $f\in(\cl_{1,p(\cdot),s}(\rn))^\ast$,
$f$ vanishes weakly at infinity and $g_{\lz,\az,s}^\ast(f)\in\vlp$;
moreover, it holds true that
$$\frac 1C\|g_{\lz,\alpha,s}^\ast(f)\|_{\vlp}\le \|f\|_{\vhs}
\le C\|g_{\lz,\alpha,s}^\ast(f)\|_{\vlp}$$
with $C$ being a positive constant independent of $f$.

The same is true if $g^\ast_{\lz,\az,s}(f)$ is replaced by
$\wz g^\ast_{\lz,(\az,\ez),s}(f)$
with $\ez\in(\az+s,\fz)$.
\end{theorem}

\begin{remark}
\begin{itemize}
\item[(i)] We point out that there exists a positive constant $C$ such that,
for all $\phi\in\ccc_{\az,s}(\rn)$, $C\phi\in\ccc_{(\az,\ez),s}(\rn)$
and hence $\phi\in\cl_{1,p(\cdot),s}(\rn)$; see Lemma \ref{l-sembed} below.
Thus, the intrinsic square functions are well defined for functionals in
$(\cl_{1,p(\cdot),s}(\rn))^\ast$. Observe that,
if $\phi\in\cs(\rn)$, then $\phi\in\cl_{1,p(\cdot),s}(\rn)$;
see also Lemma \ref{l-sembed} below. Therefore,
if $f\in(\cl_{1,p(\cdot),s}(\rn))^\ast$, then
$f\in\cs'(\rn)$ and $f$ vanishing weakly at infinity makes sense.

\item[(ii)] Recall that Liang and Yang \cite{lyts13} characterized
the Musielak-Orlicz Hardy space $H^\vz(\rn)$ in terms of the intrinsic square
functions original introduced by Wilson \cite{wilson07}.
Moreover, Liang and Yang \cite{lyts13} established the intrinsic
$g_\lz^\ast$-functions
$g_{\lz,\az,s}^\ast$ and $\wz g_{\lz,(\az,\ez),s}^\ast$ with the best known range
$\lz\in(2+2(\az+s)/n,\fz)$ via some argument similar to that used in the
proof of \cite[Theorem 4.8]{lhy}. However,
it is still unclear whether the intrinsic $g_\lz^\ast$-functions
$g_{\lz,\az,s}^\ast$ and $\wz g_{\lz,(\az,\ez),s}^\ast$,
when $\lz\in(2+2(\az+s)/n,3+2(\az+s)/n]$, can characterize
$H^{p(\cdot)}(\rn)$ or not.
Based on the same reason as in Remark \ref{r-equivalen}(ii),
we see that the method used in the proof of \cite[Theorem 1.8]{lyts13}
is not available for the present setting.

\item[(iii)] Let $p\in(0,1]$. When
\begin{equation}\label{1.x1}
p(x):=p\quad{\rm for\ all}\quad x\in\rn,
\end{equation}
then
$H^{p(\cdot)}(\rn)=H^p(\rn)$. In this case,
Theorem \ref{t-intrinsic} and Corollary \ref{c-intrinsic} coincide with
the corresponding results of the classical Hardy space $H^p(\rn)$;
see \cite[Theorem 1.6]{lyts13} and \cite[Corollary 1.7]{lyts13}.

\item[(iv)] We also point out that the method used in this article
does not work for the variable exponent Hardy space
investigated by Cruz-Uribe and Wang in
\cite{cw13}, since it strongly depends on
the locally log-H\"older continuity condition \eqref{ve1} and the
decay condition \eqref{ve2} of $p(\cdot)$.
Thus, it is still unknown whether the variable exponent Hardy space
in \cite{cw13} has any intrinsic square
function characterizations or not.
\end{itemize}
\end{remark}

\begin{definition}
Let $p(\cdot)\in\cp(\rn)$. A measure $d\mu$ on $\rr_+^{n+1}$ is called
a $p(\cdot)$-\emph{Carleson measure} if
$$\|d\mu\|_{p(\cdot)}:=\sup_{Q\subset\rn}\frac{|Q|^{1/2}}{\|\chi_Q\|_{\vlp}}
\lf\{\int_{\widehat{Q}}|d\mu(x,t)|\r\}^{1/2}<\fz,$$
where the supremum is taken over all cubes $Q\subset \rn$ and
 $\wh Q$ denotes the
\emph{tent} over $Q$, namely, $\wh Q:=\{(y,t)\in\rr_+^{n+1}:\ B(x,t)\subset Q\}$.
\end{definition}

\begin{theorem}\label{t-cm1}
Let $p(\cdot)\in\cp(\rn)$ satisfy \eqref{ve1} and \eqref{ve2}.
Assume that $p_+\in(0,1]$, $s\in\zz_+$,
$p_-\in( n/{(n+s+1)},1]$ and $\phi\in\cs(\rn)$ is a
radial function satisfying
\eqref{function1} and \eqref{function2}.
\begin{itemize}
\item[{\rm(i)}] If $b\in\cl_{1,p(\cdot),s}(\rn)$,
then $d\mu(x,t):=|\phi_t\ast b(x)|^2\,\frac{dxdt}t$ for all $(x,t)\in\rr_+^{n+1}$
is a $p(\cdot)$-Carleson measure
on $\rr_+^{n+1}$; moreover, there exists a positive constant $C$,
independent of $b$, such that
$\|d\mu\|_{p(\cdot)}\le C\|b\|_{\cl_{1,p(\cdot),s}(\rn)}$.

\item[\rm(ii)] If $b\in L_{\rm loc}^2(\rn)$ and
$d\mu(x,t):=|\phi_t\ast b(x)|^2\,\frac{dxdt}t$ for all $(x,t)\in\rr_+^{n+1}$
is a $p(\cdot)$-Carleson measure on $\rr_+^{n+1}$,
then $b\in\cl_{1,p(\cdot),s}(\rn)$
and, moreover,
there exists a positive constant $C$, independent of $b$, such that
$\|b\|_{\cl_{1,p(\cdot),s}(\rn)}\le C\|d\mu\|_{p(\cdot)}$.
\end{itemize}
\end{theorem}

In what follows, for $\az\in(0,1]$, $s\in\zz_+$, $\ez\in(0,\fz)$
and $b\in\cpss$,
the \emph{measure $\mu_b$ on}
$\rr_+^{n+1}$ is defined by setting, for all $(x,t)\in\rr_+^{n+1}$,
\begin{equation}\label{cmin}
d\mu_b(x,t):=[{\wz A_{(\az,\ez),s}(b)(x,t)}]^2\,\frac{dxdt}{t},
\end{equation}
where { $\wz A_{(\az,\ez),s}(b)$} is as in
{\eqref{intrinsic-3}} with $f$ replaced by $b$.
\begin{theorem}\label{t-cm2}
Let $\az\in(0,1]$, $s\in\zz_+$, $\ez\in(\az+s,\fz)$,
$p(\cdot)\in\cp(\rn)$ satisfy \eqref{ve1}, \eqref{ve2}, $p_+\in(0,1]$ and
$p_-\in( n/{(n+\az+s)},1]$.

\begin{itemize}
\item[\rm(i)] If $b\in\cl_{1,p(\cdot),s}(\rn)$, then
$d\mu_b$ as in \eqref{cmin} is a $p(\cdot)$-Carleson
measure on $\rr_+^{n+1}$; moreover, there exists a positive
constant $C$, independent of $b$, such that
$\|d\mu_b\|_{p(\cdot)}\le C\|b\|_{\cl_{1,p(\cdot),s}(\rn)}$.

\item[\rm(ii)] If $b\in L_{\rm loc}^2(\rn)$ and $d\mu_b$ as in \eqref{cmin} is a $p(\cdot)$-Carleson measure on $\rr_+^{n+1}$, then it follows that $b\in\cl_{1,p(\cdot),s}(\rn)$;
moreover, there exists a positive constant $C$, independent of $b$, such that
$$\|b\|_{\cl_{1,p(\cdot),s}(\rn)}\le C\|d\mu_b\|_{p(\cdot)}.$$
\end{itemize}
\end{theorem}

\begin{remark}
\begin{itemize}
\item[(i)] Fefferman and Stein \cite{fs72} shed some light on the tight connection
between BMO-functions and Carleson measures, which is the case of Theorem
\ref{t-cm1} when $s=0$ and $p(x):=1$ for all $x\in\rn$.

\item[(ii)] When $p(\cdot)$ is as in \eqref{1.x1} with $p\in(0,1]$,
Theorem \ref{t-cm1} is already known (see \cite[Theorem 4.2]{lyjmaa}).

\item[(iii)] When $p(\cdot)$ is as in \eqref{1.x1} with $p\in(0,1]$,
Theorem \ref{t-cm2} was obtained in \cite[Theorem 1.11]{lyts13}
with $p\in( n/{(n+\az+s)},1]$. Thus, the range of
$p_-$ in Theorem \ref{t-cm2} is reasonable and the best known
possible, even in the case that $p(\cdot)$ being as
in \eqref{1.x1} with $p\in(0,1]$.
\end{itemize}
\end{remark}
This article is organized as follows.

Section \ref{s2} is devoted to the proofs of Theorems
\ref{t-equivalen}, \ref{t-intrinsic}, \ref{t-intrinsic-1},
\ref{t-cm1} and \ref{t-cm2}.
To prove Theorem \ref{t-equivalen}, we establish an equivalent characterization
of $\vhs$ via the discrete Littlewood-Paley $g$-function
(see Proposition \ref{p-lpd} below)
by using the nontangential maximal function characterization of $\vhs$ obtained
in this article and the Littlewood-Paley
decomposition of $\vhs$ which was proved in \cite{ns12}. In the proof of
Theorem \ref{t-equivalen}, we also borrow some ideas
from the proofs of \cite[Theorem 2.8]{ut11} (see also \cite[Theorem 3.2]{lsuyy}).

The key tools used to prove Theorem \ref{t-intrinsic} are
the Littlewood-Paley $g$-function characterization of $\vhs$
in Theorem \ref{t-equivalen},
the atomic decomposition of $\vhs$ established in \cite{ns12}
(see also Lemma \ref{l-equi} below),
the dual space of $\vhs$, $\cl_{1,p(\cdot),s}(\rn)$, given in
\cite{ns12} and the fact that the intrinsic square functions are
pointwise comparable proved in \cite{lyts13}.
As an application of Theorems \ref{t-equivalen} and \ref{t-intrinsic},
we give the proof of Theorem \ref{t-intrinsic-1} via showing that,
for all $x\in\rn$,
the intrinsic square functions $\wz S_{(\az,\ez),s}(f)(x)$ and
$\wz g_{\lz,(\az,\ez),s}^\ast(f)(x)$ are
pointwise comparable under the assumption $\lz\in(3+2(\az+s)/n,\fz)$.

The proof of Theorem \ref{t-cm1} is similar to that of
\cite[Theorem 4.2]{lyjmaa}, which depends on atomic decomposition of the
tent space with variable exponent, the fact that the dual space of $\vhs$
is $\cpss$ (see \cite[Theorem 7.5]{ns12})
and some properties of $\cpss$. To complete the proof of Theorem \ref{t-cm1},
we first introduce the tent space with
variable exponent and obtain its atomic decomposition
in Theorem \ref{t-tentad} below. Then we give an equivalent norm of $\cpss$
via establishing a John-Nirenberg inequality for functions in $\cpss$.
At the end of Section \ref{s2}, we give the proof of Theorem \ref{t-cm2} by using
Theorem \ref{t-cm1} and some ideas from the proof of
\cite[Theorem 1.11]{lyts13}.

Finally, we make some conventions on notation.
Throughout the paper, we denote by $C$ a \emph{positive constant}
which is independent of the main parameters, but it may vary from line to line.
The \emph{symbol} $A\ls B$
means $A\le CB$. If $A\ls B$ and $B\ls A$, then we write $A\sim B$.
If $E$ is a subset of $\rn$, we denote by $\chi_E$ its
\emph{characteristic function}.
For any $x\in\rn$ and $r\in(0,\fz)$, let $B(x,r):=\{y\in\rn:\ |x-y|<r\}$
be the ball. For $\beta:=(\beta_1,\dots,\beta_n)\in\zz_+^n$,
let $\beta !:=\beta_1!\cdots\beta_n !$.
For $\alpha\in\rr$, we use $\lfloor \alpha\rfloor$
to denote the maximal integer not more than $\alpha$.
For a measurable function $f$, we use $\overline{f}$ to denote
its \emph{conjugate function}.

\section{Proofs of main results\label{s2}}

In what follows, for all $f\in\cs'(\rn)$ and $N\in\nn$,
the \emph{nontangential maximal function} $f_N^\ast$
of $f$ is defined by setting, for all $x\in\rn$,
\begin{equation}\label{d-nmf}
f_{N}^\ast(x):=\sup_{\psi\in\cf_N(\rn)}
\sup_{\gfz{t\in(0,\fz)}{|y-x|<t}}|f\ast\psi_t(y)|,
\end{equation}
where $\cf_N(\rn)$ is as in \eqref{fnd}.

The following proposition is an equivalent characterization of $\vhs$.
\begin{proposition}\label{p-vhs-p}
Let $p(\cdot)\in\cp(\rn)$ satisfy \eqref{ve1} and \eqref{ve2}, and $N$ be as in
\eqref{1nx}. Then $f\in \vhs$ if and only if $f\in\cs'(\rn)$
and $f^\ast_N\in\vlp$;
moreover, there exists a positive
constant $C$ such that, for all $f\in \vhs$,
$$C^{-1}\|f\|_{\vhs}\le\|f^\ast_N\|_{\vlp}\le C\|f\|_{\vhs}.$$
\end{proposition}
\begin{proof}
Let $f\in\cs'(\rn)$ and $f^\ast_N\in\vlp$.
Observing that, for all $x\in\rn$, $f_{N,+}^\ast(x)\le f_N^\ast(x)$,
we then conclude that
$\|f\|_{\vhs}=\|f_{N,+}^\ast\|_{\vlp}\le\|f_N^\ast\|_{\vlp}$ and
hence $f\in\vhs$. This finishes the proof of the sufficiency of Proposition \ref{p-vhs-p}.

To prove the necessity, we need to show that, for all
$f\in\vhs$, $\|f_N^\ast\|_{\vlp}\ls\|f\|_{\vhs}$. To this end,
for all $\Phi\in\cf_N(\rn)$, $x\in\rn$, $t\in(0,\fz)$ and $y\in\rn$ with
$|y-x|<t$, let, for all $z\in\rn$,
$\psi(z):=\Phi(z+(y-x)/t).$
Then we see that
\begin{eqnarray*}
\sum_{\beta\in\zz_+^n,\,|\beta|\le N}
\sup_{z\in\rn}(1+|z|)^N|D^\beta\psi(z)|=\sum_{\beta\in\zz_+^n,\,|\beta|\le N}
\sup_{z\in\rn}\lf(1+\lf|z-\frac{y-z}{t}\r|\r)^N|D^\beta\Phi(z)|
\le2^N,
\end{eqnarray*}
which implies that $2^{-N}\psi\in \cf_N(\rn)$. From this, we deduce that
\begin{eqnarray*}
|f\ast\Phi_t(y)|=|f\ast\psi_t(x)|\le 2^N f_{N,+}^\ast(x),
\end{eqnarray*}
and hence $f_{N}^\ast(x)\ls f_{N,+}^\ast(x)$ for all $x\in\rn$,
which further implies that
$$\|f_N^\ast\|_{\vlp}\ls\|f_{N,+}^\ast\|_{\vlp}\sim\|f\|_{\vhs}.$$
This finishes the proof of the necessity part and hence
Proposition \ref{p-vhs-p}.
\end{proof}

\begin{corollary}\label{c-vhs-p}
Let $p(\cdot)$ be as in Proposition \ref{p-vhs-p} and $f\in \vhs$.
Then $f$ vanishes weakly at infinity.
\end{corollary}
\begin{proof}
Observe that, for any $f\in \vhs$ with $\|f\|_{\vhs}\neq0$, $\phi\in\cs(\rn)$, $x\in\rn$, $t\in(0,\fz)$ and
$y\in B(x,t)$, $|f\ast \phi_t(x)|\ls f_N^\ast (y)$,
where $f_N^\ast$ is as in \eqref{d-nmf} with $N$ as in \eqref{1nx}.
By this and Remark \ref{r-vlp}(iv), we see that
\begin{eqnarray*}
\min\{|f\ast \phi_t(x)|^{p_+},\,|f\ast \phi_t(x)|^{p_-}\}
&&\ls{\inf_{y\in B(x,t)}\min\{[f_N^\ast (y)]^{p_+},\,[f_N^\ast (y)]^{p_-}\}}\\
&&\ls\int_{B(x,t)}\min\{[f_N^\ast (y)]^{p_+},\,[f_N^\ast (y)]^{p_-}\}
\,dy|B(x,t)|^{-1}\\
&&\ls\int_{B(x,t)}[f_N^\ast(y)]^{p(y)}\,dy|B(x,t)|^{-1}\\
&&\ls\int_\rn\lf[\frac{f_N^\ast (y)}{\|f_N^\ast\|_{\vlp}}\r]^{p(y)}
\|f_N^\ast\|_{\vlp}^{p(y)}\,dy|B(x,t)|^{-1}\\
&&\ls|B(x,t)|^{-1}\max\{\|f_N^\ast\|_{\vlp}^{p_-},\,\|f_N^\ast\|_{\vlp}^{p_+}\}
\to0,
\end{eqnarray*}
as $t\to\fz$, which implies that $f$ vanishes weakly at infinity. This finishes
the proof of Corollary \ref{c-vhs-p}.
\end{proof}
In what follows, denote by $P_{\rm poly}(\rn)$ the \emph{set of all polynomials}
on $\rn$. For $f\in\cs_\fz'(\rn)$ and
$\phi\in\cs(\rn)$ satisfying \eqref{function1} and \eqref{function2},
let
$$\sigma(f)(x):=\lf[\sum_{j\in\zz}|\phi_j\ast f(x)|^2\r]^{1/2}$$
and
$$H_{\sigma}^{p(\cdot)}(\rn):=\lf\{f\in\cs_\fz'(\rn):\
\|f\|_{H_{\sigma}^{p(\cdot)}(\rn)}:=\lf\|\sigma(f)\r\|_{\vlp}<\fz\r\}.$$

\begin{proposition}\label{p-lpd}
Let $p(\cdot)\in\cp(\rn)$ satisfy \eqref{ve1} and \eqref{ve2}. Then
$H^{p(\cdot)}(\rn)=H_{\sigma}^{p(\cdot)}(\rn)$ in the following sense:
if $f\in\vhs$, then $f\in H_{\sigma}^{p(\cdot)}(\rn)$ and there exists a positive
constant $C$ such that, for all $f\in\vhs$,
$\|f\|_{H_{\sigma}^{p(\cdot)}(\rn)}\le C\|f\|_{\vhs}$;
conversely, if $f\in H_{\sigma}^{p(\cdot)}(\rn)$, then there exists a unique
extension $\wz f \in\cs'(\rn)$ such that, for all $h\in\cs_\fz(\rn)$,
$\langle \wz f,h\rangle=\langle f,h\rangle$ and
$\|\wz f\|_{\vhs}\le C\|f\|_{H_{\sigma}^{p(\cdot)}(\rn)}$
with $C$ being a positive constant independent of $f$.
\end{proposition}
\begin{proof}
Let $f\in \vhs$.
Then $f\in\cs'(\rn)\subset\cs_\fz'(\rn)$ and, by
\cite[Theorem 5.7]{ns12} (see also \cite[Theorem 3.1]{sawa13}), we see that
$\|f\|_{H_{\sigma}^{p(\cdot)}(\rn)}\ls\|f\|_{\vhs}$ and hence
$f\in H_{\sigma}^{p(\cdot)}(\rn)$.

Conversely, let $f\in H_{\sigma}^{p(\cdot)}(\rn)$. Then $f\in\cs_\fz'(\rn)$.
From \cite[Proposition 2.3.25]{gra08},
we deduce that there exists $\wz f\in\cs'(\rn)$
such that $f-\wz f\in P_{\rm poly}(\rn)$.
By \cite[Theorem 5.7]{ns12} and the fact that
$\phi_j\ast f=\phi_j\ast \wz f$ for all $j\in\zz$ and $\phi$ as in definition
of $\sigma(f)$, we know that
\begin{equation*}
\|\wz f\|_{\vhs}\ls\|\sigma(\wz f)\|_{\vlp}\sim\|\sigma(f)\|_{\vlp}
\sim\|f\|_{H_{\sigma}^{p(\cdot)}(\rn)},
\end{equation*}
which implies that $\wz f\in \vhs$.

Suppose that there exists another extension of $f$, for example, $\wz g\in\vhs$.
Then $\wz g\in\cs'(\rn)$ and $\wz g=f$ in $\cs_\fz'(\rn)$,
which, together with \cite[Proposition 2.3.25]{gra08}, implies $\wz g-\wz f\in P_{\rm poly}(\rn)$.
From this, $\wz g-\wz f\in\vhs$ and
Corollary \ref{c-vhs-p}, we deduce that $\wz g=\wz f$ since nonzero polynomials
fail to vanish weakly at infinity. Therefore, $\wz f$ is the unique extension
of $f\in H_{\sigma}^{p(\cdot)}(\rn)$, which completes the proof of Proposition
\ref{p-lpd}.
\end{proof}

The following estimate is a special case of \cite[Lemma 3.5]{lsuyy}, which is further traced back to \cite[(2.29)]{ut11} and the argument used in the proof
of \cite[Theorem 2.6]{ut11} (see also \cite[Theorem 3.2]{lsuyy}),
\begin{lemma}\label{l-peetre}
Let $f\in\cs_\fz'(\rn)$, $N_0\in\nn$ and $\Phi\in\cs(\rn)$
satisfy \eqref{function1}
and \eqref{function2}. Then, for all $t\in[1,2]$, $a\in(0,N_0]$, $l\in\zz$
and $x\in\rn$,
it holds true that
\begin{equation*}
\lf[(\Phi_{2^{-l}t}^\ast f)_a(x)\r]^r
\le C_{(r)}\sum_{k=0}^\fz2^{-kN_0r}2^{(k+l)n}
\int_\rn\frac{|(\Phi_{k+l})_t\ast f(y)|^r}{(1+2^l|x-y|)^{ar}}\,dy,
\end{equation*}
where $r$ is an arbitrary fixed positive number and $C_{(r)}$ a positive
constant independent of $\Phi,\ f,\ l,\ t$, but may depend on $r$.
\end{lemma}
We point out that Lemma \ref{l-peetre} plays an important
role in the proof of Theorem
\ref{t-equivalen}.

The following vector-valued inequality on the boundedness of the
Hardy-Littlewood maximal operator $\cm$
on the variable Lebesgue space $\vlp$ was obtained in \cite[Corollary 2.1]{cfmp06}.

\begin{lemma}\label{l-hlmo}
Let $r\in(1,\fz)$.
Assume that $p(\cdot):\ \rn\to[0,\fz)$ is a measurable function satisfying
\eqref{ve1}, \eqref{ve2} and $1<p_-\le p_+<\fz$, then there exists a positive
constant $C_0$ such that, for all sequences $\{f_j\}_{j=1}^\fz$ of measurable
functions,
$$\lf\|\lf[\sum_{j=1}^\fz (\cm f_j)^r\r]^{1/r}\r\|_{\vlp}
\le C_0\lf\|\lf(\sum_{j=1}^\fz|f_j|^r\r)^{1/r}\r\|_{\vlp}.$$
\end{lemma}

\begin{proof}[Proof of Theorem \ref{t-equivalen}]
We first prove that, for all $f\in\cs_\fz'(\rn)$,
\begin{eqnarray}\label{equivalen-2}
\|g(f)\|_{\vlp}
&&\sim\|S(f)\|_{\vlp}\sim\|g_{a,\ast}(f)\|_{\vlp}\\
&&\sim\|\sigma(f)\|_{\vlp}
\sim\|\sigma_{a,\ast}(f)\|_{\vlp}.\noz
\end{eqnarray}

To prove \eqref{equivalen-2}, we first show that, for all $f\in\cs'_\fz(\rn)$,
\begin{equation}\label{equivalen-3}
\|g(f)\|_{\vlp}\sim\|g_{a,\ast}(f)\|_{\vlp}\
{\rm and}\ \|\sigma(f)\|_{\vlp}\sim\|\sigma_{a,\ast}(f)\|_{\vlp}.
\end{equation}
For similarity, we only give the proof for the first equivalence.
By definitions, we easily see that
$\|g(f)\|_{\vlp}\le\|g_{a,\ast}(f)\|_{\vlp}$.
Conversely, we show that $\|g_{a,\ast}(f)\|_{\vlp}\ls\|g(f)\|_{\vlp}$.
Since $a\in( n/{\min\{p_-,2\}},\fz)$,
it follows that there exists $r\in(0,\min\{p_-,2\})$ such that
$a\in(n/r,\fz)$. By Lemma \ref{l-peetre} and the Minkowski integral inequality, we find that
\begin{eqnarray*}
g_{a,\ast}(f)(x)
&&=\lf\{\sum_{j\in\zz}\int_1^2[(\phi_{2^{-j}t}^\ast f)_a(x)]^2\,
\frac{dt}t\r\}^{1/2}\\
&&\ls\lf\{\sum_{j\in\zz}\int_1^2
\lf[\sum_{k=0}^\fz2^{-kN_0r}2^{(k+j)n}
\int_\rn\frac{|(\phi_{k+j})_t\ast f(y)|^r}
{(1+2^j|x-y|)^{ar}}\,dy\r]^{\frac2r}\,\frac{dt}t\r\}^{1/2}\\
&&\ls\lf\{\sum_{j\in\zz}\lf[\sum_{k=0}^\fz2^{-kN_0r}2^{(k+j)n}
\int_\rn\frac{[\int_1^2|(\phi_{k+j})_t\ast f(y)|^2\,\frac{dt}t]^\frac r2}
{(1+2^j|x-y|)^{ar}}\,dy\r]^{\frac2r}\r\}^{1/2},
\end{eqnarray*}
which, together with the Minkowski series inequality and
Remark \ref{r-vlp}(iii), implies that
\begin{eqnarray}\label{equivalen-1}
&&\|g_{a,\ast}(f)\|_{\vlp}^r\\
&&\hs\ls\lf\|\sum_{k=0}^\fz2^{-k(N_0r-n)}\lf(\sum_{j\in\zz}2^{j\frac{2n}r}\lf[
\int_\rn\frac{[\int_1^2|(\phi_{k+j})_t\ast f(y)|^2\,\frac{dt}t]^\frac r2}
{(1+2^j|\cdot-y|)^{ar}}\,dy\r]^\frac2r\r)^\frac r2\r
\|_{L^{\frac{p(\cdot)}r}(\rn)}\noz\\
&&\hs\ls\sum_{k=0}^\fz2^{-k(N_0r-n)}\lf\|\lf\{\sum_{j\in\zz}2^{j\frac{2n}r}\lf[
\int_\rn\frac{[\int_1^2|(\phi_{k+j})_t\ast f(y)|^2\,\frac{dt}t]^\frac r2}
{(1+2^j|\cdot-y|)^{ar}}\,dy\r]^\frac2r\r\}^\frac12\r\|_{\vlp}^r\noz\\
&&\hs\ls\sum_{k=0}^\fz2^{-k(N_0r-n)}\lf\|\lf\{\sum_{j\in\zz}2^{j\frac {2n}r}\lf(
\sum_{i=0}^\fz2^{-iar}\r.\r.\r.\noz\\
&&\hs\hs\times\lf.\lf.\lf.\int_{|\cdot-y|\sim2^{i-j}}
\lf[\int_1^2|(\phi_{k+j})_t\ast f(y)|^2\,\frac{dt}t\r]^\frac r2
\,dy\r)^\frac2r\r\}^\frac12\r\|_{\vlp}^r,\noz
\end{eqnarray}
where $N_0\in\nn$ is sufficiently large and $|x-y|\sim 2^{i-j}$ means that
$|x-y|<2^{-j}$ if $i=0$, or
$2^{i-j-1}\le|x-y|<2^{i-j}$ if $i\in\nn$.
Applying the Minkowski inequality and Lemma \ref{l-hlmo},
we conclude that
\begin{eqnarray*}
&&\|g_{a,\ast}(f)\|_{\vlp}^r\\
&&\hs\ls\sum_{k=0}^\fz2^{-kN_0r+kn}\sum_{i=0}^\fz2^{-iar+in}
\lf\|\lf\{\sum_{j\in\zz}\lf[\cm\lf(\lf[\int_1^2
|(\phi_{k+j})_t\ast f|^2\,\frac{dt}t\r]
^\frac r2\r)\r]^{\frac 2r}\r\}^{\frac r2}\r\|
_{L^{\frac{p(\cdot)}{r}}(\rn)}\\
&&\hs\ls\sum_{k=0}^\fz2^{-kN_0r+kn}\sum_{i=0}^\fz2^{-iar+in}
\lf\|\lf\{\sum_{j\in\zz}\lf[\int_1^2
|(\phi_{k+j})_t\ast f|^2\,\frac{dt}t\r]^2\r\}^{\frac 12}\r\|
_{L^{p(\cdot)}(\rn)}^r\ls\|g(f)\|_{\vlp}^r,
\end{eqnarray*}
which completes the
proof of \eqref{equivalen-3}.

Next we prove that
\begin{equation}\label{equivalen-4}
\|S(f)\|_{\vlp}\sim\|g_{a,\ast}(f)\|_{\vlp}.
\end{equation}
It suffices to show that
$\|g_{a,\ast}(f)\|_{\vlp}\ls\|S(f)\|_{\vlp},$
since the inverse inequality holds true trivially.
From \cite[(3.9)]{lsuyy}, we deduce that
\begin{eqnarray*}
&&\|g_{a,\ast}(f)\|_{\vlp}\\
&&\hs=\lf\|\lf\{\sum_{j\in\zz}
\lf[\sum_{k=0}^\fz2^{-kN_0r+2(k+j)n}\r.\r.\r.\\
&&\hs\hs\times\lf.\lf.\lf.\sum_{i=0}^\fz\int_{|\cdot-y|\sim 2^{i-j}}
\lf(\int_1^2\int_{|z|<2^{-(k+j)}t}|(\phi_{k+j})_t\ast
f(y+z)|^2\,\frac{dzdt}{t}\r)
^{\frac r2}\,dy\r]^{\frac2r}\r\}\r\|_{\vlp},
\end{eqnarray*}
where $N_0\in\nn$ is sufficiently large and $|\cdot-y|\sim2^{i-j}$
is the same as
in \eqref{equivalen-1}. Then, by an argument similar to that used
in the proof of \eqref{equivalen-3}, we
conclude that $\|g_{a,\ast}(f)\|_{\vlp}\ls\|S(f)\|_{\vlp}$, which
completes the proof of \eqref{equivalen-4}.

By arguments similar to those used in the proofs of \eqref{equivalen-3},
\eqref{equivalen-4} and \cite[Theorem 2.8]{ut11}, we
conclude that
\begin{equation}\label{equivalen-5}
\|\sigma(f)\|_{\vlp}\ls\|g_{a,\ast}(f)\|_{\vlp}\ls\|\sigma_{a,\ast}(f)\|_{\vlp}.
\end{equation}

Now, from \eqref{equivalen-3}, \eqref{equivalen-4} and \eqref{equivalen-5},
we deduce that \eqref{equivalen-2} holds true, which, together with
Proposition \ref{p-lpd}, implies that
$f\in\vhs$ if and only if
$f\in\cs_\fz'(\rn)$ and $S(f)\in\vlp$; moreover,
$\|f\|_{\vhs}\sim\|S(f)\|_{\vlp}$.
This finishes the proof of Theorem \ref{t-equivalen}.
\end{proof}

\begin{proof}[Proof of Corollary \ref{c-equivalent}]
Assume $f\in\cs_\fz'(\rn)$ and $g_\lz^\ast(f)\in\vlp$.
It is easy to see that, for all $\lz\in(1,\fz)$ and $x\in\rn$,
$S(f)(x)\ls g_\lz^\ast(f)(x)$,
which, together with Theorem \ref{t-equivalen}, implies that
$f\in \vhs$ and $\|f\|_{\vhs}\sim\|S(f)\|_{\vlp}\ls\|g_\lz^\ast(f)\|_{\vlp}$.

Conversely, let $f\in\vhs$. Then $f\in\cs'(\rn)\subset\cs_\fz'(\rn)$.
By the fact that $\lz\in(1+2/{\min\{2,p_-\}},\fz)$, we see that there exists
$a\in( n/{\min\{2,p_-\}},\fz)$ such that $\lz\in(1+2a/n,\fz)$. Then, by this,
we further find that, for all $x\in\rn$,
\begin{eqnarray*}
g_\lz^\ast(f)(x)
&&=\lf\{\int_0^\fz\int_\rn\lf(\frac t{t+|x-y|}\r)^{\lz n}
|\phi_t\ast f(y)|^2\,dy\,\frac{dt}{t^{n+1}}\r\}^{1/2}\\
&&\ls\lf\{\int_0^\fz\lf[(\phi_t^\ast f)_a(x)\r]^2
\int_\rn\lf(1+\frac{|x-y|}{t}\r)^{2a-\lz n}\,dy\,\frac{dt}{t^{n+1}}\r\}^{1/2}\\
&&\sim\lf\{\int_0^\fz\lf[(\phi_t^\ast f)_a(x)\r]^2\,\frac{dt}{t}\r\}^{1/2}
\sim g_{a,\ast}(f)(x).
\end{eqnarray*}
From this and Theorem \ref{t-equivalen}, we deduce that
$$\|g_\lz^\ast(f)\|_\vlp\ls\|g_{a,\ast}(f)\|_\vlp
\sim\|f\|_{\vhs},$$
which completes the proof of Corollary \ref{c-equivalent}.
\end{proof}

To prove Theorem \ref{t-intrinsic}, we need more preparations.
The following technical lemma is essentially contained in \cite{ns12}.
\begin{lemma}\label{l-bigsball}
Let $p(\cdot)\in\cp(\rn)$ satisfy \eqref{ve1} and \eqref{ve2}.
Then there exists a positive
constant $C$ such that, for all cubes $Q_1\subset Q_2$,
\begin{equation}\label{bigsball1}
\|\chi_{Q_1}\|_{\vlp}\le C\lf(\frac{|Q_1|}{|Q_2|}\r)^{1/p_+}\|\chi_{Q_2}\|_{\vlp}
\end{equation}
and
\begin{equation*}
\|\chi_{Q_2}\|_{\vlp}\le C\lf(\frac{|Q_2|}{|Q_1|}\r)^{1/p_-}\|\chi_{Q_1}\|_{\vlp}.
\end{equation*}
\end{lemma}
\begin{proof}
For similarity, we only show \eqref{bigsball1}.
Let $z_0\in Q_1$.
If $\ell(Q_2)\le1$, then, by \cite[Lemma 2.2(1)]{ns12} and its proof, we see that
\begin{eqnarray*}
\frac{\|\chi_{Q_1}\|_{\vlp}}{\|\chi_{Q_2}\|_{\vlp}}
\sim\lf(\frac{|Q_1|}{|Q_2|}\r)^{\frac1{p(z_0)}}
\ls\lf(\frac{|Q_1|}{|Q_2|}\r)^{\frac1{p_+}}.
\end{eqnarray*}
If $\ell(Q_1)\ge1$, then by \cite[Lemma 2.2(2)]{ns12}, we find that
\begin{eqnarray*}
\frac{\|\chi_{Q_1}\|_{\vlp}}{\|\chi_{Q_2}\|_{\vlp}}
\sim\lf(\frac{|Q_1|}{|Q_2|}\r)^{\frac1{p_\fz}}
\ls\lf(\frac{|Q_1|}{|Q_2|}\r)^{\frac1{p_+}},
\end{eqnarray*}
where $p_\fz$ is as in \eqref{ve2}.
If $\ell(Q_1)<1<\ell(Q_2)$, then by \cite[Lemma 2.2]{ns12}, we know that
\begin{eqnarray*}
\frac{\|\chi_{Q_1}\|_{\vlp}}{\|\chi_{Q_2}\|_{\vlp}}
\sim\frac{|Q_1|^{1/p(z_0)}}{|Q_2|^{1/p_\fz}}
\ls\lf(\frac{|Q_1|}{|Q_2|}\r)^{\frac1{p_+}},
\end{eqnarray*}
which completes the proof of \eqref{bigsball1} and hence Lemma \ref{l-bigsball}.
\end{proof}
The following Lemma \ref{l-cm1} comes from \cite[p.38]{mw80}.
\begin{lemma}\label{l-cm1}
Let $g\in L_{\rm loc}^1(\rn)$, $s\in\zz_+$ and $Q$ be a cube in $\rn$. Then
there exists a positive constant $C$, independent of $g$ and $Q$, such that
$$\sup_{x\in Q}|P_Q^sg(x)|\le \frac C{|Q|}\int_Q|g(x)|\,dx.$$
\end{lemma}
\begin{lemma}\label{l-sembed}
Let $\az\in(0,1]$, $s\in\zz_+$ and
$\ez\in(\az+s,\fz)$. Assume that $p(\cdot)\in\cp(\rn)$
satisfies \eqref{ve1}, \eqref{ve2}
and $p_-\in(n/(n+\az+s),1]$.
If $f\in\ccc_{(\az,\ez),s}(\rn)$ or $\cs(\rn)$,
then $f\in\cl_{1,p(\cdot),s}(\rn)$.

\end{lemma}
\begin{proof}
For similarity, we only give the proof for $\ccc_{(\az,\ez),s}(\rn)$.
For any $f\in\ccc_{(\az,\ez),s}(\rn)$,  $x\in\rn$ and cube $Q:=Q(x_0,r)\subset\rn$
with $(x_0,r)\in\rr_+^{n+1}$, let
$$p_Q(x):=\sum_{|\beta|\le s}
\frac{D^\beta f(x_0)}{\beta !}(x-x_0)^\beta\in\cp_s(\rn).$$
Then, from Lemma \ref{l-cm1} and Taylor's remainder theorem,
we deduce that, for any $x\in Q$, there exists $\xi(x)\in Q$ such that
\begin{eqnarray}\label{sembed1}
\int_Q|f(x)-P_Q^sf(x)|\,dx
&&\le\int_Q|f(x)-p_Q(x)|\,dx+\int_Q|P_Q^s(p_Q-f)(x)|\,dx\\
&&\ls\int_Q|f(x)-p_Q(x)|\,dx\noz\\
&&{\sim}\int_Q\lf|\sum_{|\beta|=s}\frac{D^\beta f(\xi(x))-D^\beta f(x_0)}{\beta !}
(x-x_0)^\beta\r|\,dx.\noz
\end{eqnarray}

Now, if $|x_0|+r\le1$, namely, $Q\subset Q(0,\sqrt{n})$,
then, by Lemma \ref{l-bigsball},
\eqref{sembed1}, \eqref{intrinsic-2} and the fact that $p_-\in(n/(n+\az+s),1]$,
we see that
\begin{eqnarray}\label{sembed2}
&&\frac1{\|\chi_Q\|_{\vlp}}\int_Q|f(x)-P_Q^sf(x)|\,dx\\
&&\hs\ls\lf\{\sup_{x,\,y\in\rn,x\neq y}\sum_{|\beta|=s}
\frac{|D^\beta f(x)-D^\beta f(y)|}{|x-y|^\alpha}\r\}
\frac{1}{\|\chi_Q\|_{\vlp}}\int_Q|\xi(x)-x_0|^\az|x-x_0|^s\,dx\noz\\
&&\hs\ls|Q|^{1+(\az+s)/n-1/p_-}\frac{|Q(0,\sqrt{n})|^{1/p_-}}
{\|\chi_{Q(0,\sqrt n)}\|_{\vlp}}\ls1.\noz
\end{eqnarray}

If $|x_0|+r>1$ and $|x_0|\le 2r$,
then $r>1/3$ and $|Q|\sim|Q(0,\sqrt n(|x_0|+r))|$.
From Lemma \ref{l-cm1} and $|f(x)|\le(1+|x|)^{-n-\ez}$ for all $x\in\rn$, we deduce that
\begin{eqnarray}\label{sembed3}
\quad\quad&&\frac1{\|\chi_Q\|_{\vlp}}\int_Q|f(x)-P_Q^sf(x)|\,dx\\
&&\hs\ls\frac1{\|\chi_Q\|_{\vlp}}\int_Q|f(x)|\,dx
\!\ls\!\sup_{y\in\rn}\lf[(1+|y|)^{n+\vez}|f(y)|\r]\frac1{\|\chi_Q\|_{\vlp}}
\int_Q\frac{1}{(1+|x|)^{n+\vez}}\,dx\noz\\
&&\hs\ls\lf[\frac{|Q(0,\sqrt n(|x_0|+r))|}{|Q|}\r]^{1/p-}
\frac1{\|\chi_{Q(0,\sqrt n(|x_0|+r))}\|_{\vlp}}\ls1.\noz
\end{eqnarray}

If $|x_0|+r>1$ and $|x_0|>2r$, then, for all $x\in Q$, it holds that
$1\ls|x|\sim|x_0|$. By this, \eqref{sembed1}, Lemma \ref{l-bigsball}
and \eqref{intrinsic-2},
we find that
\begin{eqnarray}\label{sembed4}
&&\frac1{\|\chi_Q\|_{\vlp}}\int_Q|f(x)-P_Q^sf(x)|\,dx\\
&&\hs\ls\frac{1}{\|\chi_Q\|_{\vlp}}
\int_Q|\xi(x)-x_0|^\az(1+|x_0|)^{-n-\ez}
|x-x_0|^s\,dx\noz\\
&&\hs\ls\frac{|Q|^{1+\frac{\az+s}n}}{\|\chi_Q\|_{\vlp}}(|x_0|+r)^{-n-\ez}\noz\\
&&\hs\ls\frac{|Q|^{1+(\az+s)/n}}{(|x_0|+r)^{n+\ez}}
\lf(\frac{|Q(0,\sqrt n(|x_0|+r))|}{|Q|}\r)^{\frac1{p_-}}
\frac1{\|\chi_{Q(0,\sqrt n(|x_0|+r))}\|_{\vlp}}\ls1.\noz
\end{eqnarray}

Combining \eqref{sembed2}, \eqref{sembed3}
and \eqref{sembed4}, we see that $f\in\cl_{1,p(\cdot),s}(\rn)$, which
completes the proof of Lemma \ref{l-sembed}.
\end{proof}
\begin{remark}
{ We point out that, from the proof of Lemma \ref{l-sembed},
we know that $\ccc_{(\alpha,\vez),s}(\rn)$ and $\cs(\rn)$
are continuously embedding into $\cl_{1,p(\cdot),s}(\rn)$,
which, in the case of $s=0$ and $p(x):=1$ for all $x\in\rn$,
was proved in \cite[Proposition 2.1]{ny11}.
Indeed, by the proof of Lemma \ref{l-sembed}, we see that,
for all $f\in\ccc_{(\alpha,\vez),s}(\rn)$ or $\cs(\rn)$,
\begin{eqnarray*}
\|f\|_{\cl_{1,p(\cdot),s}(\rn)}
&&\ls \sup_{x\in\rn}(1+|x|)^{n+\vez}|f(x)|\\
&&\quad+\sup_{x,\,y\in\rn,x\neq y}\sum_{|\beta|=s}
\lf\{\lf[\frac1{(1+|x|)^{n+\vez}}+\frac1{(1+|y|)^{n+\vez}}\r]^{-1}
\frac{|D^\beta f(x)-D^\beta f(y)|}{|x-y|^\alpha}\r\};
\end{eqnarray*}
moreover, if $f\in\ccc_{(\alpha,\vez),s}(\rn)$, then
$\|f\|_{\cl_{1,p(\cdot),s}(\rn)}\ls 1$; if $f\in\cs(\rn)$, then
$$\|f\|_{\cl_{1,p(\cdot),s}(\rn)}
\ls \sup_{x\in\rn}\sum_{\beta\in\zz_+^n,\,|\beta|\le s+1}
(1+|x|)^{n+\vez}|D^\beta f(x)|.$$
In this sense, $\ccc_{(\alpha,\vez),s}(\rn)$ and $\cs(\rn)$
are continuously embedding into $\cl_{1,p(\cdot),s}(\rn)$.}
\end{remark}

Now we recall the atomic Hardy space with variable exponent introduced
by Nakai and Sawano \cite{ns12}.
Let $p(\cdot)\in \cp(\rn)$, $s\in( n/{p_-}-n-1,\fz)\cap\zz_+$ and
$q\in[1,\fz]$ satisfy that $q\in[p_+,\fz)$.
Recall that a measurable function $a$ on $\rn$ is
called a $(p(\cdot),q,s)$-\emph{atom} if
it satisfies the following three conditions:

\begin{itemize}
\item[(i)] supp\,$a\subset Q$ for some $Q\subset\rn$;
\item[(ii)] $\|a\|_{\vlp}\le\frac{|Q|^{1/q}}{\|\chi_Q\|_{\vlp}}$;
\item[(iii)] $\int_\rn a(x)x^\beta\,dx=0$ for any $\beta\in\zz_+^n$ and $|\beta|\le s$.
\end{itemize}

The \emph{atomic Hardy space with variable} $p(\cdot)$,
denoted by $H_{\rm atom}^{p(\cdot),q}(\rn)$, is
defined to be the set of all $f\in\cs'(\rn)$ that can be represented as a sum of
multiples of $(p(\cdot),q,s)$-atoms, namely, $f=\sum_{j}\lz_j a_j$ in $\cs'(\rn)$,
where, for each $j$, $\lz_j$ is a nonnegative number and $a_j$ is a
$(p(\cdot),q,s)$-atom supported in some cube $Q_j$ with the property
$$\int_\rn\lf\{\sum_j\lf(\frac{\lz_j\chi_{Q_j}(x)}{\|\chi_{Q_j}\|_{\vlp}}\r)^
{p^\ast}\r\}^{\frac{p(x)}{p^\ast}}\,dx<\fz$$
with $p^\ast:=\min\{p_-,1\}$.
The \emph{norm} of $f\in H_{\rm atom}^{p(\cdot),q}(\rn)$ is defined by
$$\|f\|_{H_{\rm atom}^{p(\cdot),q}(\rn)}
:=\inf\lf\{\ca(\{\lz_j\}_{j},\{Q_j\}_{j}):\
f=\sum_{j}\lz_ja_j\ {\rm in}\ \cs'(\rn)\r\},$$
where the infimum is taken over all decompositions of $f$ as above
and
$$\ca(\{\lz_j\}_{j},\{Q_j\}_{j})
:=\inf\lf\{\lz\in(0,\fz):\ \int_\rn\lf(\sum_j\lf[\frac{\lz_j
\chi_{Q_j}(x)}{\lz\|\chi_{Q_j}\|_{\vlp}}\r]^{p^\ast}
\r)^{\frac{p(x)}{p^\ast}}\,dx\le1\r\}.$$

The following conclusion is just \cite[Lemma 4.11]{ns12}.

\begin{lemma}\label{l-ineq}
Let $p(\cdot)\in\cp(\rn)$ satisfy \eqref{ve1} and \eqref{ve2}.
Then there exist $\beta_0\in(0,1)$ and a positive constant $C$ such that, if
$q\in(0,\fz)$ satisfies $1/q\in(0,-{\log_2\beta_0}/{(n+1)})$,
then, for all sequences $\{\lz_j\}_j$ of nonnegative numbers,
measurable functions $\{b_j\}_j$ and cubes $\{Q_j\}_j$
satisfying supp $b_j\subset Q_j$ and $\|b_j\|_{L^{q}(Q_j)}\neq0$ for each $j$,
$$\lf\|\lf\{\sum_j\lf(\frac{\lz_j|b_j||Q_j|^{1/q}}
{\|b_j\|_{L^{q}(Q_j)}\|\chi_{Q_j}\|_{\vlp}}\r)^{p^\ast}\r\}^\frac1{p^\ast}
\r\|_{\vlp}\le C\ca(\{\lz_j\}_j,\{Q_j\}_j).$$
\end{lemma}

Let $q\in[1,\fz]$ and $s\in\zz_+$.
Denote by $L_{\rm comp}^{q,s}(\rn)$
the \emph{set of all functions $f\in L^\fz(\rn)$ with compact}
and \
$$L_{\rm comp}^{q,s}(\rn):=\lf\{f\in L_{\rm comp}^q(\rn):
\ \int_\rn f(x)x^\alpha\,dx=0,\ |\alpha|\le s\r\}.$$
As point out in \cite[p.\,3707]{ns12}, $L_{\rm comp}^{q,s}(\rn)$ is dense in
$H^{p(\cdot),q}_{\rm atom}(\rn)$.
The conclusions of the following Lemmas \ref{l-equi}  and \ref{l-dual}
were, respectively, just \cite[Theorems 4.6]{ns12} and \cite[Theorem 7.5]{ns12},
which play key roles in the proof of
Theorem \ref{t-equivalen}.

\begin{lemma}\label{l-equi}
 Let $q\in[1,\fz]$ and $p(\cdot)\in\cp(\rn)$ satisfy \eqref{ve1}, \eqref{ve2}
and $p_+\in(0,q)$.
Assume that $q$ is as in Lemma \ref{l-ineq}.
Then $\vhs=H^{p(\cdot),q}_{\rm atom}(\rn)$
with equivalent quasi-norms.
\end{lemma}

\begin{lemma}\label{l-dual}
Let $p(\cdot)\in\cp(\rn)$ satisfy \eqref{ve1},
\eqref{ve2}, $p_+\in(0,1]$, $q\in(p_+,\fz)$
and $s\in( n/{p_-}-n-1,\fz)\cap\zz_+$.
Then the dual space of $H^{p(\cdot),q}_{\rm atom}(\rn)$, denoted by
$(H^{p(\cdot),q}_{\rm atom}(\rn))^\ast$, is $\cl_{q',p(\cdot),s}(\rn)$
in the following sense:
for any $b\in\cl_{q',p(\cdot),s}(\rn)$, the linear functional
\begin{equation}\label{dual-x}
\ell_b(f):=\int_\rn b(x)f(x)\,dx,
\end{equation}
initial defined for all $f\in L_{\rm comp}^{q,s}(\rn)$,
has a bounded extension to
$H^{p(\cdot),q}_{\rm atom}(\rn)$;
conversely, if $\ell$ is a bounded linear functional on
$H^{p(\cdot),q}_{\rm atom}(\rn)$, then $\ell$ has the form as in \eqref{dual-x}
with a unique $b\in\cl_{q',p(\cdot),s}(\rn)$.

Moreover,
$$\|b\|_{\cl_{q',p(\cdot),s}(\rn)}
\sim\|\ell_b\|_{(H^{p(\cdot),q}_{\rm atom}(\rn))^\ast},$$
where the implicit positive constants are independent of $b$.
\end{lemma}
The following Lemma \ref{l-intri-eq} is
just from \cite[Theorem 2.6]{lyts13},
which, in the case when $s=0$,
was first proved by Wilson \cite[Theorem 2]{wilson07}.
\begin{lemma}\label{l-intri-eq}
Let $\alpha\in(0,1]$, $s\in\zz_+$ and $\ez\in(\max\{\az,s\},\fz)$.
Then there exists a positive constant $C$ such that, for all $f$
satisfying \eqref{intrinsic-1} and $x\in\rn$,
$$\frac1C g_{\alpha,s}(f)(x)\le\wz g_{(\alpha,\ez),s}(f)(x)
\le C g_{\az,s}(f)(x).$$
\end{lemma}
The following Lemma \ref{l-intri-bou} is a special case of
\cite[Proposition 3.2]{lyts13}.
\begin{lemma}\label{l-intri-bou}
Let $\az\in(0,1]$, $s\in\zz_+$ and $q\in(1,\fz)$.
Then there exists a positive constant $C$ such that, for all
measurable functions $f$,
$$\int_\rn[g_{\az,s}(f)(x)]^q\,dx\le C\int_\rn|f(x)|^q\,dx.$$
\end{lemma}

Now we come to give a proof of Theorem \ref{t-intrinsic}.

\begin{proof}[Proof of Theorem \ref{t-intrinsic}]
For $\ez\in(\az+s,\fz)$, by Lemma \ref{l-intri-eq}, we see that
$g_{\az,s}(f)$ and $\wz g_{(\az,\ez),s}(f)$ are pointwise comparable.
Thus, to prove Theorem \ref{t-intrinsic},
it suffices to show that the conclusion of Theorem \ref{t-intrinsic}
holds true for the intrinsic square function $g_{\az,s}(f)$.

Let $f\in(\cl_{1,p(\cdot),s}(\rn))^\ast$ vanish weakly at infinity
and $g_{\az,s}(f)\in\vlp$.
Then, by Lemma \ref{l-sembed}, we find that
$f\in\cs'(\rn)\subset\cs_\fz'(\rn)$.
Notice that, for all $x\in\rn$,
$g(f)(x)\ls \wz g_{(\az,\ez),s}(f)(x)\sim g_{\az,s}(f)(x)$
(see Lemma \ref{l-intri-eq}),
it follows that $g(f)\in\vlp$.
From this and Theorem \ref{t-equivalen}, we deduce that
there exists a distribution $\wz f\in\cs'(\rn)$ such that
$\wz f=f$ in $\cs_\fz'(\rn)$, $\wz f\in\vhs$ and
$\|\wz f\|_{\vhs}\ls\|g(f)\|_{\vlp}$,
which, together with Corollary \ref{c-vhs-p} and the fact that
$f$ vanishes weakly at infinity, implies that $f=\wz f$ in $\cs'(\rn)$ and hence
$$\|f\|_{\vhs}\sim\|\wz f\|_{\vhs}\ls\|g(f)\|_{\vlp}\ls\|g_{\az,s}(f)\|_{\vlp}.$$
This finishes the proof of the sufficiency of Theorem \ref{t-equivalen}.

It remains to prove the necessity. Let $f\in\vhs$.
Then, by Corollary \ref{c-vhs-p},
we see that $f$ vanishes weakly at infinity
and, by Lemmas \ref{l-equi} and \ref{l-dual},
we have $f\in(\cl_{1,p(\cdot),s}(\rn))^\ast$.
If $q\in(1,\fz)$ is as in Lemma \ref{l-ineq},
then, by Lemma \ref{l-equi}, we know that
there exist a sequence $\{\lz_j\}_j$ of nonnegative numbers
and a sequence $\{a_j\}_j$ of $(p(\cdot),q,s)$-atoms,
with supp $a_j\subset Q_j$ for all $j$,
such that $f=\sum_{j}\lz_ja_j$ in $\cs'(\rn)$ and also in $\vhs$ and, moreover
\begin{equation}\label{estimate-3}
\ca(\{\lz_j\}_j,\{Q_j\}_j)\ls\|f\|_{\vhs}.
\end{equation}
Thus, by Lemma \ref{l-sembed}, we find that,
for all $\phi\in\ccc_{(\alpha,\vez),s}(\rn)$,
$f\ast \phi=\sum_j\lz_j a_j\ast \phi$ pointwise and hence, for all $x\in\rn$,
$g_{\az,s}(f)(x)\le\sum_j\lz_jg_{\az,s}(a_j)(x)$.

Now, for a $(p(\cdot),q,s)$-atom $a$ with supp $a\subset Q:=Q(x_0,r)$,
we estimate $g_{\az,s}(a)$. By
Lemma \ref{l-intri-bou}, we find that
\begin{eqnarray}\label{estimate-2}
\|g_{\az,s}(a)\|_{L^q(2\sqrt{n}Q)}
\ls\|a\|_{L^q(\rn)}
\ls\frac{|Q|^{1/q}}{\|Q\|_{\vlp}},
\end{eqnarray}
here and hereafter, $2\sqrt n Q$ denotes the cube with the center same as $Q$ but with
the side length $2\sqrt n$ times $Q$.

On the other hand, for all $x\notin 2\sqrt{n}Q$,
by the vanishing moment condition of $a$ and \eqref{condition1},
together with Taylor's remainder theorem, we see that
\begin{eqnarray}\label{estimate-1}
|a\ast \phi_t(x)|&&=\frac1{t^n}\lf|\int_\rn a(y)\lf[\phi\lf(\frac{x-y}{t}\r)-
\sum_{|\beta|\le s}\frac{D^\beta\phi(\frac{x-x_0}{t})}{\beta!}
\lf(\frac{x_0-y}{t}\r)^\beta\r]\,dy\r|\\
&&\ls\int_\rn|a(y)|\frac{|y-x_0|^{\az+s}}{t^{n+\az+s}}\,dy
\ls\frac1{\|\chi_Q\|_{\vlp}}\lf(\frac rt\r)^{n+\az+s}.\noz
\end{eqnarray}
Notice that supp $\phi\subset\{x\in\rn:\ |x|\le1\}.$ If $x\notin 2\sqrt{n}Q$
and $\phi_t\ast a(x)\neq0$, then, there exists a $y\in Q$ such that
$|x-y|/t\le1$ and hence
$t\ge|x-y|\ge|x-x_0|-|x_0-y|>|x-x_0|/2.$
From this and \eqref{estimate-1}, we deduce that
\begin{eqnarray*}
g_{\az,s}(a)(x)
&&=\lf\{\int_0^\fz\lf[\sup_{\phi\in\ccc_{\az,s}(\rn)}
|a\ast\phi_t(x)|\r]^2\,\frac{dt}{t}\r\}^{1/2}\\
&&\ls\frac1{\|\chi_Q\|_{\vlp}}r^{n+\az+s}
\lf\{\int_{\frac{|x-x_0|}2}^\fz t^{-2(n+\az+s)}\,dt\r\}^{1/2}\noz\\
&&\ls\frac1{\|\chi_Q\|_{\vlp}}\lf(\frac{r}{|x-x_0|}\r)^{n+\az+s}
\ls\frac{[\cm(\chi_Q)(x)]^{\frac{n+\az+s}n}}{\|\chi_Q\|_{\vlp}},\noz
\end{eqnarray*}
which implies that
\begin{eqnarray}\label{estimate-5}
\qquad\|g_{\az,s}(f)\|_{\vlp}&&\ls\lf\|\sum_j\lz_jg_{\az,s}(a_j)\chi_{2\sqrt{n}Q_j}\r\|_{\vlp}
+\lf\|\sum_j\lz_j\frac{[\cm(\chi_{Q_j})]^{\frac{n+\az+s}n}}
{\|\chi_{Q_j}\|_{\vlp}}\r\|_{\vlp}\\
&&=:{\rm I}_1+{\rm I}_2.\noz
\end{eqnarray}

For ${\rm I}_1$, by taking $b_j:={ g_{\az,s}(a_j)\chi_{2\sqrt{n}Q_j}}$
for each $j$ in Lemma \ref{l-ineq},
\eqref{estimate-2} and Lemma \ref{l-bigsball}, we conclude that
\begin{eqnarray}\label{estimate-4}
{\rm I}_1
&&\ls\lf\|\sum_j\frac{\lz_jb_j|Q_j|^{\frac1q}}
{\|b_j\|_{L^q(2\sqrt{n}Q_j)}\|\chi_{2\sqrt{n}Q_j}\|_{\vlp}}\r\|_{\vlp}\\
&&\ls\lf\|\lf\{\sum_j\lf(\frac{\lz_jb_j|Q_j|^{\frac1q}}
{\|b_j\|_{L^q(2\sqrt{n}Q_j)}\|\chi_{2\sqrt{n}Q_j}\|_{\vlp}}\r)
^{p^\ast}\r\}^{1/p^\ast}\r\|_{\vlp}\!\!\ls\!\!\ca(\{\lz_j\}_j,\{Q_j\}_j).\noz
\end{eqnarray}

For ${\rm I}_2$, letting $\theta:={(n+\az+s)}/n$,
by Lemma \ref{l-hlmo} and $p_-\in( n/{(n+\az+s)},\fz)$, we find that
\begin{eqnarray*}
{\rm I}_2
&&\ls\lf\|\lf\{\sum_j\frac{\lz_j[\cm(\chi_{Q_j})]^\theta}
{\|\chi_{Q_j}\|_{\vlp}}\r\}^\frac1{\theta}\r\|_{L^{\theta p(\cdot)}
(\rn)}^{\theta}
\ls\lf\|\sum_j\frac{\lz_j\chi_{Q_j}}
{\|\chi_{Q_j}\|_{\vlp}}\r\|_{L^{p(\cdot)}(\rn)}\\
&&\ls\lf\|\lf\{\sum_j\lf(\frac{\lz_j\chi_{Q_j}}
{\|\chi_{Q_j}\|_{\vlp}}\r)^{p^\ast}\r\}^{\frac1{p^\ast}}\r\|_{L^{p(\cdot)}(\rn)}
\sim\ca(\{\lz_j\}_j,\{Q_j\}_j).
\end{eqnarray*}
From this, together with \eqref{estimate-3}, \eqref{estimate-5}
and \eqref{estimate-4},
we deduce that
$$\|g_{\az,s}(f)\|_{\vlp}\ls\|f\|_{\vhs},$$
which completes the proof of Theorem \ref{t-intrinsic}.
\end{proof}

For $s\in\zz_+$, $\alpha\in(0,1]$ and $\ez\in(0,\fz)$,
let $\ccc_{(\az,\ez),s}(y,t)$, with $y\in\rn$ and $t\in(0,\fz)$, be the
\emph{family of functions} $\psi\in C^s(\rn)$ such that,
for all $\gamma\in\zz_+^n$, $|\gamma|\le s$ and $x\in\rn$,
$|D^\gamma\psi(x)|\le t^{-n-|\gamma|}(1+|y-x|/t)^{-n-\ez}$,
$\int_\rn\psi(x)x^\gamma\,dx=0$
and, for all $x_1,\ x_2\in\rn$, $\nu\in\zz_+^n$ and $|\nu|=s$,
$$|D^\nu\psi(x_1)-D^\nu\psi(x_2)|
\le\frac{|x_1-x_2|^\az}{t^{n+\gamma+\az}}
\lf[\lf(1+\frac{|y-x_1|}{t}\r)^{-n-\ez}+\lf(1+\frac{|y-x_2|}{t}\r)^{-n-\ez}\r].$$

The proof of Theorem \ref{t-intrinsic-1} needs the following Lemma \ref{l-claim},
whose proof is trivial, the details being omitted.
\begin{lemma}\label{l-claim}
Let $s\in\zz_+$, $\alpha\in(0,1]$, $\ez\in(0,\fz)$ and $f$ be
a measurable function
satisfying $\eqref{intrinsic-1}$.
\begin{itemize}
\item[\rm(i)] For any $y\in\rn$ and $t\in(0,\fz)$, it holds true that
$$\wz A_{(\az,\ez),s}(f)(y,t)=\sup_{\psi\in\ccc_{(\az,\ez),s}(y,t)}
\lf|\int_\rn\psi(x)f(x)\,dx\r|.$$

\item[\rm(ii)] If $t_1,\ t_2\in(0,\fz)$, $t_1<t_2$, $y\in\rn$ and
$\psi\in\ccc_{(\az,\ez),s}(y,t_1)$, then
$(\frac{t_1}{t_2})^{n+s+\az}\psi\in\ccc_{(\az,\ez),s}(y,t_2)$.
\end{itemize}
\end{lemma}

\begin{proof}[Proof of Theorem \ref{t-intrinsic-1}]
If $f\in(\cl_{1,p(\cdot),s}(\rn))^\ast$,
$\wz g_{\lz,(\az,\ez),s}^\ast(f)\in\vlp$ and
$f$ vanishes weakly at infinity, then, by Lemma \ref{l-sembed}, we see that
$f\in\cs'(\rn)$ and, by the fact that, for all $x\in\rn$,
$$g_{\lz}^\ast(f)(x)\ls g_{\lz,\az,s}^\ast(f)(x)\ls\wz
g_{\lz,(\az,\ez),s}^\ast(f)(x)$$
and Theorem \ref{t-equivalen}, we further know that
$f\in\vhs$ and
$$\|f\|_{\vhs}\ls\|g_\lz^\ast(f)\|_{\vlp}\ls
\|g_{\lz,\az,s}^\ast(f)\|_{\vlp}\ls\|\wz g_{\lz,(\az,\ez),s}^\ast(f)\|_{\vlp}.$$
This finishes the proof of the sufficiency of Theorem \ref{t-intrinsic-1}.

Next we prove the necessity of Theorem \ref{t-intrinsic-1}.
Let $f\in\vhs$. Then, as in the proof of Theorem \ref{t-intrinsic},
we see that $f\in(\cl_{1,p(\cdot),s}(\rn))^\ast$ and $f$
vanishes weakly at infinity.
For all $x\in\rn$, we have
\begin{eqnarray}\label{gast-1}
&&[\wz g_{\lz,(\az,\ez),s}^\ast(f)(x)]^2\\
&&\hs=\int_0^\fz\int_{|y-x|<t}\lf(\frac t{t+|x-y|}\r)^{\lz n}
[\wz A_{(\az,\ez),s}(f)(y,t)]^2\,\frac{dy\,dt}{t^{n+1}}\noz\\
&&\hs\hs+\sum_{k=1}^\fz\int_0^\fz\int_{2^{k-1}t\le|y-x|<2^kt}
\cdots\,\frac{dy\,dt}{t^{n+1}}\noz\\
&&\hs\ls [\wz S_{(\az,\ez),s}(f)(x)]^2+\sum_{k=1}^\fz2^{-k\lz n}
\int_0^\fz\int_{|y-x|<2^kt}
[\wz A_{(\az,\ez),s}(f)(y,t)]^2\,\frac{dy\,dt}{t^{n+1}}\noz\\
&&\hs\sim[\wz S_{(\az,\ez),s}(f)(x)]^2+
\sum_{k=1}^\fz2^{-k\lz n}2^{kn}\int_0^\fz\int_{|y-x|<t}
[\wz A_{(\az,\ez),s}(f)(y,2^{-k}t)]^2\,\frac{dy\,dt}{t^{n+1}}.\noz
\end{eqnarray}
By Lemma \ref{l-claim}, we find that, for all $k\in\nn$ and $(y,t)\in\urn$,
\begin{eqnarray*}
\wz A_{(\az,\ez),s}(f)(y,2^{-k}t)
&&=\sup_{\psi\in \ccc_{(\az,\ez),s}(y,2^{-k}t)}\lf|\int_{\rn}\psi(x)f(x)\,dx\r|\\
&&\le2^{k(n+s+\az)}\sup_{\wz\psi\in \ccc_{(\az,\ez),s}(y,t)}
\lf|\int_{\rn}\wz\psi(x)f(x)\,dx\r|
=2^{k(n+s+\az)}\wz A_{(\az,\ez),s}(f)(y,t),
\end{eqnarray*}
which, together with \eqref{gast-1} and $\lz\in(3+{2(s+\az)}/n,\fz)$,
implies that
\begin{eqnarray*}
[\wz g_{\lz,(\az,\ez),s}^\ast(f)(x)]^2
\ls\sum_{k=0}^\fz2^{-k\lz n}2^{k(3n+2s)}[\wz S_{(\az,\ez),s}(f)(x)]^2
\sim[\wz S_{(\az,\ez),s}(f)(x)]^2.
\end{eqnarray*}
From this, together with Theorem \ref{c-intrinsic},
we deduce that $\wz g_{\lz,(\az,\ez),s}^\ast(f)\in\vlp$ and
$$\|g_{\lz,\az,s}^\ast(f)\|_{\vlp}
\ls\|\wz g_{\lz,(\az,\ez),s}^\ast(f)\|_{\vlp}
\ls\|\wz S_{(\az,\ez),s}(f)\|_{\vlp}\ls\|f\|_{\vhs},$$
which completes the proof of Theorem \ref{t-intrinsic-1}.
\end{proof}

To prove Theorem \ref{t-cm1}, we first introduce the tent space with
variable exponent. For all measurable functions $g$
on $\rr_+^{n+1}$ and $x\in\rn$, define
$$\ca(g)(x)
:=\lf\{\int_0^\fz\int_{\{y\in\rn:\ |y-x|<t\}}|g(y,t)|^2\,
\frac{dy\,dt}{t^{n+1}}\r\}^{1/2}.$$
Recall that a measurable function $g$ is said to belong to the \emph{tent space}
$T_2^p(\rr_+^{n+1})$ with $p\in(0,\fz)$, if
$\|g\|_{T_2^p(\rr_+^{n+1})}:=\|\ca(g)\|_{L^p(\rn)}<\fz$.

Let $p(\cdot)\in\cp(\rn)$ satisfy \eqref{ve1} and \eqref{ve2}.
In what follows, we denote by $T^{p(\cdot)}_2(\rr_+^{n+1})$ the \emph{space} of all
measurable functions $g$ on $\rr_+^{n+1}$ such that $\ca(g)\in\vlp$ and,
for any $g\in T^{p(\cdot)}_2(\rr_+^{n+1})$, its \emph{quasi-norm} is defined by
$$\|g\|_{T^{p(\cdot)}_2(\rr_+^{n+1})}
:=\|\ca(g)\|_{\vlp}
:=\inf\lf\{\lz\in(0,\fz)
:\ \int_\rn\lf(\frac{\ca(g)(x)}{\lz}\r)^{p(x)}\,dx\le1\r\}.$$

Let $p\in(1,\fz)$. A function $a$ on $\rr_+^{n+1}$ is called a
$(p(\cdot),p)$-\emph{atom} if
there exists a cube $Q\subset\rn$ such that ${\rm supp}\,a\subset\wh Q$
and $\|a\|_{T_2^p(\rr_+^{n+1})}\le |Q|^{1/p}\|\chi_Q\|_{\vlp}^{-1}$.
Furthermore, if $a$ is a $(p(\cdot),p)$-atom for all
 $p\in(1,\fz)$, we then call $a$
a $(p(\cdot),\fz)$-\emph{atom}.

For functions in the space $T^{p(\cdot)}_2(\urn)$,
we have the following atomic decomposition.

\begin{theorem}\label{t-tentad}
Let $p(\cdot)\in\cp(\rn)$ satisfy \eqref{ve1} and \eqref{ve2}. Then, for any
$f\in T^{p(\cdot)}_2(\urn)$, there exist $\{\lz_j\}_{j}\subset\cc$
and a sequence $\{a_j\}_{j}$ of $(p(\cdot),\fz)$-atoms such that, for almost every
$(x,t)\in\urn$, $ f(x,t)=\sum_{j}\lz_ja_j(x,t).$
Moreover, there exists a positive constant $C$ such that, for all
$f\in T^{p(\cdot)}_2(\urn)$, $\ca^\ast(\{\lz_j\}_{j},\{Q_j\}_{j})
\le C\|f\|_{T^{p(\cdot)}_2(\urn)}$, where
\begin{eqnarray}\label{tentad}
\ca^\ast(\{\lz_j\}_{j},\{Q_j\}_{j}):=\inf\lf\{\lz\in(0,\fz):\ \sum_{j}\int_{Q_j}
\lf[\frac{\lz_j}{\lz\|\chi_{Q_j}\|_{\vlp}}\r]^{p(x)}\,dx\le1\r\}
\end{eqnarray}
and, for each $j$, $Q_j$ appears in the support of $a_j$.
\end{theorem}

\begin{remark}\label{r-adp-1}
Assume that $p_+\in(0,1]$. Then, by \cite[Remark 4.4]{ns12}, we know that,
for any sequences $\{\lz_j\}_{j}$ of nonnegative numbers and cubes
$\{Q_j\}_{j}$,
$\sum_{j}\lz_j\le \ca^\ast(\{\lz_j\}_{j},\{Q_j\}_{j}).$
\end{remark}
The proof of Theorem \ref{t-tentad} is similar to
that of \cite[Theorem 3.2]{hyy} (see also \cite[Theorem 3.1]{jyjfa10}).
To this end, we need some known facts as follows
(see, for example, \cite[Theorem 3.1]{jyjfa10}).

Let $F$ be a closed subset of $\rn$ and $O:=\rn\backslash F=:F^\complement$.
Assume that $|O|<\fz$.
For any fixed $\gamma\in(0,1)$, $x\in\rn$ is said to have the
\emph{global $\gamma$-density} with respect to $F$ if, for all $t\in(0,\fz)$,
${|B(x,t)\cap F|}/{|B(x,t)|}\ge\gz$. Denote by $F_\gz^\ast$ the
\emph{set of all such $x$} and let $O_\gz^\ast:=(F_\gz^\ast)^\complement$.
Then
$$O_\gamma^\ast=\{x\in\rn:\ \cm(\chi_O)(x)>1-\gamma\}$$
 is open, $O\subset O_\gamma^\ast$
and there exists a positive constant $C_{(\gamma)}$,
depending on $\gamma$, such that
$|O_\gamma^\ast|\le C_{(\gamma)}|O|$.
For any $\nu\in(0,\fz)$ and $x\in\rn$, let
$\Gamma_\nu(x):=\{(y,t)\in\urn:\ |x-y|<\nu t\}$ be the
\emph{cone of aperture $\nu$
with vertex $x\in\rn$} and $\Gamma(x):=\Gamma_1(x)$. Denote by $\mr_\nu F$ the
\emph{union of all cones with vertices in $F$}, namely,
$\mr_\nu F:=\cup_{x\in F}\Gamma_\nu(x)$.

The following Lemma \ref{l-tentad1} is just \cite[Lemma 3.1]{jyjfa10}.
\begin{lemma}\label{l-tentad1}
Let $\nu,\ \eta\in(0,\fz)$. Then there exist positive constants $\gamma\in(0,1)$
and $C$ such that, for any closed subset $F$ of $\rn$
whose complement has finite measure, and any nonnegative measurable
function $H$ on $\urn$,
$$\int_{\mr_\nu(F_\gamma^\ast)}H(y,t)t^n\,dydt
\le C \int_F\lf\{\int_{\Gamma_\eta}H(y,t)\,dydt\r\}\,dx,$$
where $F_\gamma^\ast$ denotes the set of points in $\rn$ with
the global $\gamma$-density with respect to $F$.
\end{lemma}

\begin{proof}[Proof of Theorem \ref{t-tentad}]
Assume that $f\in T^{p(\cdot)}_2(\urn)$. For any $k\in\zz$, we let
$$O_k:=\lf\{x\in\rn:\ \ca(f)(x)>2^k\r\}$$
and $F_k:=O_k^\complement$. Since $f\in T^{p(\cdot)}_2(\urn)$, for
each $k$, $O_k$ is an open set of $\rn$ and $|O_k|<\fz$.
Let $\gamma\in(0,1)$ be as in Lemma \ref{l-tentad1} with $\eta=1=\nu$.
In what follows, we denote $(F_k)_\gamma^\ast$ and $(O_k)_\gamma^\ast$ simply by
$F_k^\ast$ and $O_k^\ast$.
By the proof of \cite[Theorem 3.2]{hyy}, we know that
${\rm supp}\,f\subset(\cup_{k\in\zz}\wh{O_k^\ast}\cup E)$, where $E\subset\urn$
satisfies that $\int_E\frac{dy\,dt}{t}=0$.

For each $k\in\zz$, considering the Whitney decomposition of the
open set of $O_k^\ast$,
we obtain a set $I_k$ of indices and
a family $\{Q_{k,j}\}_{j\in I_k}$ of closed cubes
 with disjoint interiors such that

(i) $\cup_{j\in I_k}Q_{k,j}=O_k^\ast$ and, if $i\neq j$, then
$\mathring Q_{k,j}\cap \mathring Q_{k,i}=\emptyset$, where $\mathring E$ denotes the
\emph{interior} of the set $E$;

(ii) $\sqrt n\ell(Q_{k,j})\le{\rm dist}(Q_{k,j},(O_k^\ast)^\complement)
\le 4\sqrt n\ell(Q_{k,j})$, where $\ell(Q_{k,j})$ denotes the \emph{side-length}
of $Q_{k,j}$ and ${\rm dist}(Q_{k,j},(O_k^\ast)^\complement)
:=\inf\{|z-w|:\ z\in Q_{k,j},\ w\in(O_k^\ast)^\complement\}$.

Now, for each $j\in I_k$, let $R_{k,j}$ be the \emph{cube with the same center as
$Q_{k,j}$
and with the radius ${11\sqrt n}/2$-times $\ell(Q_{k,j})$}. Set
$$A_{k,j}:=\wh{R_{k,j}}\cap(Q_{k,j}\times(0,\fz))
\cap(\wh{O_k^\ast}\backslash\wh{O_{k+1}^\ast}),$$
$$a_{k,j}:=2^{-k}\|\chi_{R_{k,j}}\|_{\vlp}^{-1}f\chi_{A_{k,j}}$$
and $\lz_{k,j}:=2^{k}\|\chi_{R_{k,j}}\|_{\vlp}$.
Notice that $(Q_{k,j}\times(0,\fz))
\cap(\wh{O_k^\ast}\backslash\wh{O_{k+1}^\ast})\subset\wh{R_{k,j}}$.
From this and ${\rm supp}\,f\subset(\cup_{k\in\zz}\wh{O_k^\ast}\cup E)$,
we deduce that
$f=\sum_{k\in\zz}\sum_{j\in I_k}\lz_{k,j}a_{k,j}$
almost everywhere on $\urn$.

Next we first show that, for each $k\in\zz$ and $j\in I_k$, $a_{k,j}$ is
a $(p(\cdot),\fz)$-atom support in $\wh{R_{k,j}}$.
Let $p\in(1,\fz)$ and $h\in T_2^{p'}(\urn)$ with $\|h\|_{T_2^{p'}(\urn)}\le1$.
Since $A_{k,j}\subset (\wh{O_{k+1}^\ast})^\complement=\mr_1(F^\ast_{k+1})$,
by Lemma \ref{l-tentad1} and the H\"older inequality, we have
\begin{eqnarray*}
|\langle a_{k,j},h\rangle|
&&:=\lf|\int_{\urn}a_{k,j}(y,t)\chi_{A_{k,j}}(y,t)h(y,t)\,\frac{dy\,dt}{t}\r|\\
&&\ls\int_{F_{k+1}}\int_{\Gamma(x)}|a_{k,j}(y,t)h(y,t)|\,\frac{dy\,dt}{t^{n+1}}dx
\ls\int_{F_{k+1}}\ca(a_{k,j})(x)\ca(h)(x)\,dx\\
&&\ls2^{-k}\|\chi_{R_{k,j}}\|_{\vlp}^{-1}
\lf\{\int_{(3R_{k,j})\cap F_{k+1}}
[\ca(f)(x)]^p\,dx\r\}^{1/p}\|h\|_{T_2^{p'}(\urn)}\\
&&\ls|R_{k,j}|^{1/p}\|\chi_{R_{k,j}}\|_{\vlp}^{-1},
\end{eqnarray*}
which, together with $(T_2^p(\urn))^\ast=T_2^{p'}(\urn)$ (see \cite{cms85}),
where $(T_2^p(\urn))^\ast$ denotes the \emph{dual space} of $T_2^{p}(\urn)$,
implies that $\|a_{k,j}\|_{T_2^{p}(\urn)}
\ls|R_{k,j}|^{1/p}\|\chi_{R_{k,j}}\|_{\vlp}^{-1}$.
Thus, $a_{k,j}$ is a $(p(\cdot),p)$-atom support
in $\wh{R_{k,j}}$ up to a harmless
constant for all $p\in(1,\fz)$ and hence a $(p(\cdot),\fz)$-atom up to
a harmless constant.

Finally, we prove that $\ca^\ast(\{\lz_j\}_{j},\{Q_j\}_{j})
\ls\|f\|_{T^{p(\cdot)}_2(\urn)}$.
By the fact that $\chi_{R_{k,j}}\ls\cm(\chi_{Q_{k,j}}^r)$ for any $r\in(0,p_-)$,
we know that
\begin{eqnarray*}
&&\ca^\ast(\{\lz_{k,j}\},\{R_{k,j}\})\\
&&\hs\le\lf\|\lf\{\sum_{k\in\zz}\sum_{j\in I_k}
\frac{|\lz_{k,j}|^{p_-}\chi_{R_{k,j}}}{\|\chi_{R_{k,j}}\|
_{\vlp}^{p_-}}\r\}^{\frac 1{p_-}}\r\|_{\vlp}
=\lf\|\lf\{\sum_{k\in\zz}\sum_{j\in I_k}\lf(2^k\chi_{R_{k,j}}
\r)^{p_-}\r\}^{\frac1{p_-}}\r\|_{\vlp}\\
&&\hs\ls\lf\|\lf\{\sum_{k\in\zz}\sum_{j\in I_k}
\lf[\cm(2^{kr}\chi_{Q_{k,j}}^r)(x)\r]^{\frac{p_-}r}\r\}^{1/p_-}\r\|_{\vlp},
\end{eqnarray*}
which, together with Lemma \ref{l-hlmo} and the
Whitney decomposition of $O_k^\ast$,
implies that
\begin{eqnarray*}
\ca^\ast(\{\lz_{k,j}\},\{R_{k,j}\})
\ls\lf\|\lf\{\sum_{k\in\zz}\sum_{j\in I_k}\lf(2^k
\chi_{Q_{k,j}}\r)^{p_-}\r\}^{1/p_-}
\r\|_{\vlp}\!\sim\lf\|\lf\{\sum_{k\in\zz}\lf(2^k\chi_{O_k^\ast}\r)^{p_-}\r\}^{1/p_-}
\r\|_{\vlp}.
\end{eqnarray*}
From the fact that $\chi_{O_k^\ast}\ls\cm(\chi_{O_k}^r)$ with $r\in(0,p_-)$
and Lemma \ref{l-hlmo} again, we further deduce that
\begin{eqnarray*}
&&\ca^\ast(\{\lz_{k,j}\},\{R_{k,j}\})\\
&&\hs\ls\lf\|\lf\{\sum_{k\in\zz}\lf(\cm(2^{kr}
\chi_{O_k}^r)\r)^{p_-/r}\r\}^{1/p_-}
\r\|_{\vlp}
\ls\lf\|\lf\{\sum_{k\in\zz}\lf(2^k\chi_{O_k}\r)^{p_-}\r\}^{1/p_-}
\r\|_{\vlp}\\
&&\hs\sim\lf\|\lf\{\sum_{k\in\zz}\lf(2^k
\chi_{O_k\backslash O_{k+1}}\r)^{p_-}\r\}^{1/p_-}
\r\|_{\vlp}
\sim\|\ca(f)\|_{\vlp}\sim\|f\|_{T^{p(\cdot)}_2(\urn)},
\end{eqnarray*}
which completes the proof of Theorem \ref{t-tentad}.
\end{proof}

To prove Theorem \ref{t-cm1}, we also need following technical lemmas.
\begin{lemma}\label{l-cm3}
Let $Q:=Q(x_0,\delta)\subset\rn$, $\vez\in(n(1/{p_-}-1),\fz)$,
$p(\cdot)\in\cp(\rn)$ satisfy \eqref{ve1} and \eqref{ve2}, and
$s\in( n/{p_-}-n-1,\fz)\cap\zz_+$.
Then there exists a positive constant $C$ such that, for all $f\in\cpss$,
$$\int_\rn\frac{\delta^\vez|f(x)-P_{Q}^sf(x)|}
{\delta^{n+\vez}+|x-x_0|^{n+\delta}}\,dx\le C
\frac{\|\chi_{Q}\|_{\vlp}}{|Q|}\|f\|_{\cpss}.$$
\end{lemma}
To prove Lemma \ref{l-cm3}, we need the following Lemma \ref{l-cm2} which was
proved in \cite[Lemma 6.5]{ns12}.

\begin{lemma}\label{l-cm2}
Let $p(\cdot)\in\cp(\rn)$ and $q\in[1,\fz]$. Assume that $p(\cdot)$ satisfies
\eqref{ve1} and \eqref{ve2}, and $s\in( n/{p_-}-n-1,\fz)\cap\zz_+$.
Then there exists a positive constant $C$ such that,
for all $Q\in\cq$, $j\in\zz$ and $f\in\cps$,
$$\lf\{\frac1{|2^j Q|}\int_{2^jQ}\lf|f(x)-P_Q^sf(x)\r|^q\,dx\r\}^{1/q}
\le C2^{jn(\frac 1{p_-}-1)}\frac{\|\chi_Q\|_{\vlp}}{|Q|}\|f\|_{\cps},$$
where $2^jQ$ denotes the cube with the same center as $Q$ but $2^j$ times
side-length of $Q$.
\end{lemma}

\begin{proof}[Proof of Lemma \ref{l-cm3}]
For any $k\in\zz$, let $Q_k:=2^{k}Q$, namely, $Q_k$ has the same center with
$Q$ but with $2^k$ times side-length of $Q$. Then we have
\begin{eqnarray*}
{\rm I}
:=&&\int_\rn\frac{\delta^\vez|f(x)-P_{Q}^sf(x)|}
{\delta^{n+\vez}+|x-x_0|^{n+\delta}}\,dx\\
=&&\lf(\int_Q+\sum_{k=0}^\fz\int_{Q_{k+1}\backslash Q_k}\r)
\frac{\delta^\vez|f(x)-P_{Q}^sf(x)|}
{\delta^{n+\vez}+|x-x_0|^{n+\delta}}\,dx\\
\ls&&\frac1{|Q|}\int_Q|f(x)-P_Q^sf(x)|\,dx
+\sum_{k=0}^\fz(2^k\delta)^{-n-\vez}\delta^\vez
\int_{Q_{k+1}}|f(x)-P_Q^sf(x)|\,dx\\
\ls&&\frac{\|\chi_Q\|_{\vlp}}{|Q|}\|f\|_{\cpss}\!+\!\sum_{k=1}^\fz\frac{2^{-k(n+\vez)}}{|Q|}
\int_{Q_k}\lf[|f(x)-P_{Q_k}^sf(x)|+|P_{Q_k}^sf(x)-P_Q^sf(x)|\r]\,dx.
\end{eqnarray*}
By Lemmas \ref{l-cm1} and \ref{l-cm2}, we find that, for all $x\in Q_k$,
\begin{eqnarray*}
|P_{Q_k}^sf(x)-P_Q^sf(x)|
&&=|P_{Q_k}^s(f-P_Q^sf)(x)|\ls\frac1{|Q_k|}\int_{Q_k}|f(x)-P_Q^sf(x)|\,dx\\
&&\ls2^{k(\frac n{p_-}-n)}\frac{\|\chi_Q\|_{\vlp}}{|Q|}\|f\|_{\cpss},
\end{eqnarray*}
which, together with $\vez\in(n(1/{p_-}-1),\fz)$, implies that
\begin{eqnarray*}
{\rm I}
&&\ls\frac{\|\chi_Q\|_{\vlp}}{|Q|}\|f\|_{\cpss}
+\sum_{k=1}^\fz2^{-k(\vez+n-\frac n{p_-})}
\frac{\|\chi_Q\|_{\vlp}}{|Q|}\|f\|_{\cpss}\\
&&\sim\frac{\|\chi_Q\|_{\vlp}}{|Q|}\|f\|_{\cpss}.
\end{eqnarray*}
This finishes the proof of Lemma \ref{l-cm3}.
\end{proof}
Next we establish a John-Nirenberg inequality for functions in $\cpss$.
\begin{lemma}\label{l-jnineq}
Let $p(\cdot)\in\cp(\rn)$ satisfy \eqref{ve1} and \eqref{ve2},  $f\in\cpss$ with
$s\in( n/{p_-}-n-1,\fz)\cap\zz_+$.
Assume that $p_+\in(0,1]$. Then there exist positive constants
$c_1$ and $c_2$, independent of $f$,
such that, for all cubes $Q\subset\rn$ and $\lz\in(0,\fz)$,
$$|\{x\in Q:\ |f(x)-P_Q^sf(x)|>\lz\}|
\le c_1{\rm exp}\lf\{-\frac{c_2|Q|\lz}
{\|f\|_{\cpss}\|\chi_Q\|_{\vlp}}\r\}|Q|.$$
\end{lemma}

\begin{proof}
Let $f\in\cpss$ and a cube $Q\subset\rn$. Without loss of generality,
we may assume that $\|f\|_{\cpss}\|\chi_Q\|_{\vlp}=|Q|$.
Otherwise, we replace $f$ by
${f|Q|}/{[\|f\|_{\cpss}\|\chi_Q\|_{\vlp}]}.$
Thus, to show the conclusion of Lemma \ref{l-jnineq}, it suffices to show that
\begin{equation}\label{jnineq-3}
|\{x\in Q:\ |f(x)-P_Q^sf(x)|>\lz\}|
\le c_1{\rm exp}\lf\{-c_2\lz\r\}|Q|.
\end{equation}

For any $\lz\in(0,\fz)$ and cube $R\subset Q$,
let
${\rm I}(\lz, R):=|\{x\in R:\ |f(x)-P_R^sf(x)|>\lz\}|$
and
\begin{equation}\label{jnineq-4}
\cf(\lz,Q):=\sup_{R\subset Q}\frac{{\rm I}(\lz, R)}{|R|}.
\end{equation}
Then it is easy to see that $\cf(\lz,Q)\le1$.
From Lemma \ref{l-bigsball}, $\|f\|_{\cpss}\|\chi_Q\|_{\vlp}=|Q|$
and $p_+\in(0,1]$, we deduce that there exist a positive constant
$c_0$ such that,
for any cube $R\subset Q$,
$$\frac1{R}\int_R|f(x)-P_R^sf(x)|\,dx
\le\frac{\|\chi_R\|_{\vlp}}{|R|}\|f\|_{\cpss}\le c_0.$$

Applying the Calder\'on-Zygmund decomposition of $|f-P_R^sf|$ at
height $\sigma\in(c_0,\fz)$ on the cube $R$, there exists a family $\{R_k\}_k$ of cubes
of $R$ such that
$|f(x)-P_R^sf(x)|\le \sigma$ for almost every $x\in R\backslash(\cup_kR_k)$,
$R_k\cap R_j=\emptyset$ if $k\neq j$
and, for all $k$,
$\sigma<\int_{R_k}|f(x)-P_R^sf(x)|\,dx/{|R_k|}\le 2^n\sigma$.
From this, we deduce that
\begin{equation}\label{jnineq-1}
\sum_{k}|R_k|\le\frac1{\sigma}\sum_k\int_{R_k}|f(x)-P_R^sf(x)|\,dx
\le\frac1{\sigma}\int_R|f(x)-P_R^sf(x)|\,dx\le\frac{c_0}{\sigma}|R|.
\end{equation}
If $\lz\in(\sigma,\fz)$, then, for almost every $x\in R\backslash(\cup_kR_k)$,
$|f(x)-P_R^sf(x)|\le\sigma<\lz$ and hence
\begin{eqnarray}\label{jnineq-2}
{\rm I}(\lz,R)
&&\le\sum_{k}|\{x\in R_k:\ |f(x)-P_R^sf(x)|>\lz\}|\\
&&\le\sum_k{\rm I}(\lz-\eta,R_k)+
\sum_k|\{x\in R_k:\ |P_{R_k}^sf(x)-P_R^sf(x)|>\eta\}|=:{\rm I}_1+{\rm I}_2,\noz
\end{eqnarray}
where $\eta\in(0,\lz)$ is determined later. For I$_1$, by \eqref{jnineq-4} and
\eqref{jnineq-1}, we have
\begin{eqnarray}\label{jnineq-5}
{\rm I}_1\le\sum_k\cf(\lz-\eta,Q)|R_k|\le\frac{C_0}\sigma\cf(\lz-\eta,Q)|R|.
\end{eqnarray}
For I$_2$, by Lemma \ref{l-cm1}, we find that there exists a positive
constant $C_1$ such that, for any $x\in R_k$,
\begin{equation*}
|P_{R_k}^sf(x)-P_{R}^sf(x)|=|P_{R_k}^s(f-P_R^sf)(x)|
\le\frac{C_1}{|R_k|}\int_{R_k}|f(x)-P_R^sf(x)|\,dx
\le 2^nC_1\sigma.
\end{equation*}

Now, let $\sigma:=2c_0$ and $\eta=2^nC_1\sigma$.
Then, when $\lz\in(\eta,\fz)$, I$_2=0$, which, together with \eqref{jnineq-2}
and \eqref{jnineq-5},
implies that ${\rm I}(\lz,R)\le\cf(\lz-\eta,Q)|R|/2$ for all $R\subset Q$.
Thus, it follows that $\cf(\lz,Q)\le \cf(\lz-\eta,Q)/2$.
If $m\in\nn$ satisfies $m\eta<\lz\le(m+1)\eta$, then
$$\cf(\lz,Q)\le\frac12\cf(\lz-\eta,Q)\le\cdots\le\frac1{2^m}\cf(\lz-m\eta,Q).$$
From $\cf(\lz-m\eta,Q)\le 1$ and $m\ge\lz/\eta-1$, we deduce that
$$\cf(\lz,Q)\le 2^{-m}\le 2^{1-\lz/\eta}=2e^{(-\frac1\eta\log2)\lz}.$$
Therefore, when $\lz\in(\eta,\fz)$, we conclude that \eqref{jnineq-3}
holds true with
$c_1:=2$ and $c_2:=(\log2)/\eta$.
On the other hand, when $\lz\in(0,\eta)$, \eqref{jnineq-3} holds true trivially.
This finishes the proof of Lemma \ref{l-jnineq}.
\end{proof}
By the H\"older inequality and Lemma \ref{l-jnineq}, we immediately obtain
the following
Corollary \ref{c-jnineq}, the details being omitted.
\begin{corollary}\label{c-jnineq}
{ Let $p(\cdot)$, $s$ be as in Lemma \ref{l-jnineq} and $r\in(1,\fz)$. Then
$f\in\cpss$ if and only if $f\in \cl_{r,p(\cdot),s}(\rn)$.}
\end{corollary}
Now we prove Theorem \ref{t-cm1}.
\begin{proof}[Proof of Theorem \ref{t-cm1}]
We first prove (i). Let $b\in\cpss$. For any $Q_0:=Q(x_0,r)$, write
\begin{equation}\label{cm-z}
b=P_{2Q_0}^sb+(b-P_{2Q_0}^sb)\chi_{2Q_0}+
(b-P_{2Q_0}^sb)\chi_{\rn\backslash(2Q_0)}
=:b_1+b_2+b_3.
\end{equation}

For $b_1$, since $\int_\rn\phi(x)x^\gamma\,dx=0$ for all $\gamma\in\zz_+^n$,
we see that, for all $t\in(0,\fz)$, $\phi_t\ast b_1\equiv0$
and hence
\begin{equation}\label{cm-y}
\int_{\wh {Q_0}}|\phi_t\ast b_1(x)|^2\frac{dxdt}{t}=0.
\end{equation}

For $b_2$, by the fact that the boundedness of the square function
$g(f)$ on $L^2(\rn)$ (see, for example, \cite[p.\,356, Exercise 5.1.4]{gra08}),
we find that
\begin{eqnarray*}
\int_{\wh {Q_0}}|\phi_t\ast b_2(x)|^2\frac{dxdt}{t}
&&\le \int_{\urn}|\phi_t\ast b_2(x)|^2\frac{dxdt}{t}
\ls\|b_2\|_{L^2(\rn)}^2\sim\int_{2Q_0}|b(x)-P_{2Q_0}^sb(x)|^2\,dx,
\end{eqnarray*}
which, together with Lemma \ref{l-bigsball} and Corollary \ref{c-jnineq},
implies that
\begin{eqnarray}\label{cm-x}
\frac{|Q_0|^{1/2}}{\|\chi_{Q_0}\|_{\vlp}}
\lf\{\int_{\wh {Q_0}}|\phi_t\ast b_2(x)|^2\frac{dxdt}{t}\r\}^{1/2}
\ls\|b\|_{\cpss}.
\end{eqnarray}

For $b_3$, let $\vez$ be as in Lemma \ref{l-cm3}.
Then, for all $(x,t)\in \wh {Q_0}$, we have
\begin{eqnarray*}
|\phi_t\ast b_3(x)|
&&\ls\int_{\rn\backslash(2Q_0)}
\frac{t^\vez}{(t+|x-y|)^{n+\vez}}|b(y)-P_{2Q_0}^s(y)|\,dy\\
&&\ls\int_{\rn\backslash(2Q_0)}
\frac{t^\vez}{(t+|x_0-y|)^{n+\vez}}|b(y)-P_{2Q_0}^s(y)|\,dy\\
&&\ls\frac{t^\vez}{r^\vez}\frac{\|\chi_{Q_0}\|_{\vlp}}{|Q_0|}\|b\|_{\cpss},
\end{eqnarray*}
which implies that
\begin{eqnarray}\label{cm-x2}
\frac{|Q_0|}{\|\chi_{Q_0}\|_{\vlp}}\lf\{\frac1{|Q_0|}\int_{\wh {Q_0}}
|\phi_t\ast b_3(x)|^2\,\frac{dxdt}{t}\r\}^{1/2}
\ls\|b\|_{\cpss}.
\end{eqnarray}
From this, \eqref{cm-z}, \eqref{cm-y} and \eqref{cm-x}, we deduce that
\begin{equation*}
\frac{|Q_0|}{\|\chi_{Q_0}\|_{\vlp}}\lf\{\frac1{|Q_0|}\int_{\wh {Q_0}}
|\phi_t\ast b(x)|^2\,\frac{dxdt}{t}\r\}^{1/2}
\ls\|b\|_{\cpss},
\end{equation*}
which, together with the arbitrariness of $Q_0\subset\rn$,
implies that $d\mu$ is a $p(\cdot)$-Carleson measure on $\urn$ and
$\|d\mu\|_{p(\cdot)}\ls\|b\|_{\cpss}$.

Next, we prove (ii).
To this end, let $f\in L_{\rm comp}^{\fz,s}(\rn)$. Then, by $f\in L^\fz(\rn)$
with compact support,
$b\in L_{\rm loc}^2(\rn)$ and the Plancherel formula, we conclude that
\begin{equation}\label{cm-x1}
\lf|\int_\rn f(x)\overline{b(x)}\,dx\r|\sim\lf|\int_{\urn}
\phi_t\ast f(x)\overline{\phi_t\ast b(x)}\,
\frac{dxdt}{t}\r|.
\end{equation}
Moreover, from $f\in \vhs$ and Theorem \ref{t-equivalen},
we deduce that $\phi_t\ast f\in T^{p(\cdot)}_2(\urn)$, which, combined with
Theorem \ref{t-tentad}, implies that there exist $\{\lz_j\}_j\subset \cc$ and
a sequence $\{a_j\}_j$ of $(p(\cdot),\fz)$-atoms with
supp$a_j\subset \widehat{Q_j}$ such that
$\phi_t\ast f(x)=\sum_j\lz_ja_j(x,t)$ almost everywhere. By this, \eqref{cm-x1},
the H\"older inequality and Remark \ref{r-adp-1}, we find that
\begin{eqnarray*}
\lf|\int_\rn f(x)\overline{b(x)}\,dx\r|
&&\le \sum_j|\lz_j|\int_{\urn}|a_j(x,t)||\phi_t\ast b(x)|\,\frac{dxdt}{t}\\
&&\le \sum_j|\lz_j|\lf\{\int_{\wh {Q_j}}|a_j(x,t)|^2\r\}^{1/2}
\lf\{\int_{\wh{Q_j}}|\phi_t\ast b(x)|^2\,\frac{dxdt}{t}\r\}^{1/2}\\
&&\ls\sum_j|\lz_j|\frac{|Q_j|^{1/2}}{\|\chi_{Q_j}\|_{\vlp}}
\lf\{\int_{\wh{Q_j}}|\phi_t\ast b(x)|^2\,\frac{dxdt}{t}\r\}^{1/2}\\
&&\ls\|f\|_{\vhs}\|d\mu\|_{p(\cdot)}\ls\|d\mu\|_{p(\cdot)},
\end{eqnarray*}
which, together with \cite[Theorem 7.5]{ns12} and the fact that
$L_{\rm comp}^{\fz,s}(\rn)$ is dense in
$\vhs$, implies that $\|b\|_{\cpss}\ls\|d\mu\|_{p(\cdot)}$ and hence
completes the proof of Theorem \ref{t-cm1}.
\end{proof}
We conclude this section by giving the proof of Theorem \ref{t-cm2}.
\begin{proof}[Proof of Theorem \ref{t-cm2}]
From Theorem \ref{t-cm1}(ii) and the fact that, for all $(x,t)\in\rr_+^{n+1}$,
$|\phi_t\ast b(x)|\ls \wz A_{(\az,\ez),s}(b)(x,t)$ with $\phi$
as in Theorem \ref{t-cm1},
we deduce that the conclusion of Theorem \ref{t-cm2}(ii) holds true.

It therefore remains to prove (i).
Let $b\in\cl_{1,p(\cdot),s}(\rn)$. Then, for any cube $Q_0\subset\rn$, write
\begin{equation*}
b=P_{2Q_0}^sb+(b-P_{2Q_0}^sb)\chi_{2Q_0}+(b-P_{2Q_0}^sb)
\chi_{\rn\backslash(2Q_0)}
=:b_1+b_2+b_3.
\end{equation*}

For $b_1$, since $\int_\rn\phi(x)x^\gamma\,dx=0$ for
$\phi\in \ccc_{(\az,\ez),s}(\rn)$
and $\gamma\in\zz_+^n$ with $|\gamma|\le s$,
we see that, for all $t\in(0,\fz)$, it holds true that $\phi_t\ast b_1\equiv0$
and hence
\begin{equation}\label{cm-in-1}
\int_{\wh {Q_0}}[\wz A_{(\az,\ez),s}(b_1)(x,t)]^2\,\frac{dxdt}{t}=0.
\end{equation}

For $b_2$, from Lemmas \ref{l-intri-bou} and \ref{l-intri-eq}, we deduce that
$$\int_{\wh {Q_0}}[\wz A_{(\az,\ez),s}(b_2)(x,t)]^2\,\frac{dxdt}{t}
\ls\|b_2\|_{L^2(\rn)}^2\sim\int_{2Q_0}|b(x)-P_{Q_0}^sb(x)|^2\,dx,$$
which, together with Corollary \ref{c-jnineq}, implies that
\begin{equation}\label{cm-in-2}
\frac{|Q_0|^{1/2}}{\|\chi_{Q_0}\|_{\vlp}}
\lf\{\int_{\wh {Q_0}}|\wz A_{(\az,\ez),s}( b_2)(x,t)|^2\frac{dxdt}{t}\r\}
\ls\|b\|_{\cl_{1,p(\cdot),s}(\rn)}.
\end{equation}
By an argument similar to that used in the proof of \eqref{cm-x2}, we find that
\begin{equation*}
\frac{|Q_0|^{1/2}}{\|\chi_{Q_0}\|_{\vlp}}
\lf\{\int_{\wh {Q_0}}|\wz A_{(\az,\ez),s}( b_3)(x,t)|^2\frac{dxdt}{t}\r\}
\ls\|b\|_{\cl_{1,p(\cdot),s}(\rn)}.
\end{equation*}
From this, combining \eqref{cm-in-1} and \eqref{cm-in-2}, we conclude that
$d\mu_b$ is a $p(\cdot)$-Carleson measure on $\urn$ and
$\|d\mu_b\|_{p(\cdot)}\ls\|b\|_{\cl_{1,p(\cdot),s}(\rn)}$,
which completes the proof of Theorem \ref{t-cm2}.
\end{proof}

\vspace{0.2cm}

\noindent\textbf{Acknowledgement.}
{ The authors would like to thank the referees for their several valuable remarks
which improve the presentation of this article.} This project is
supported by the National Natural Science Foundation of China
(Grant Nos.~11171027 \& 11361020),
the Specialized Research Fund for the Doctoral Program of Higher Education
of China (Grant No. 20120003110003) and the Fundamental Research
Funds for Central Universities of China (Grant No.~2012LYB26).
{ The second author, Dachun Yang, is the corresponding author of this article.}

\end{document}